\documentclass[12pt]{amsart}
\usepackage{a4wide,amsfonts,graphicx,
amssymb,stmaryrd,amscd,amsmath,latexsym,amsbsy}
\numberwithin{equation}{section}
\theoremstyle{plain}
\newtheorem{thm}{Theorem}[section]

\newtheorem{prop}[thm]{Proposition}

\newtheorem{defi}[thm]{Definition}
\newtheorem{lem}[thm]{Lemma}
\newtheorem{cor}[thm]{Corollary}
\newtheorem{eg}[thm]{Example}

\theoremstyle{remark}
\newtheorem{rema}[thm]{Remark}

\title[The c-function expansion of a basic hypergeometric function]
{The c-function expansion of a basic hypergeometric function
associated to root systems}
\author{J.V. Stokman}
\address{J.V. Stokman, Korteweg-de Vries 
Institute for Mathematics, University of Amsterdam,
Science Park 904, 1098 XH Amsterdam, The Netherlands \&
IMAPP, Radboud University Nijmegen, Heyendaalseweg 135, 6525 AJ Nijmegen, The Netherlands}
\email{j.v.stokman@uva.nl}
\begin{document}
\begin{abstract}
We derive an explicit $c$-function expansion of a basic hypergeometric
function associated to root systems. The
basic hypergeometric function in question was constructed
as explicit series expansion in symmetric
Macdonald polynomials by Cherednik in case
the associated twisted affine root system is reduced. Its 
construction was extended to the nonreduced case by the author.
It is a meromorphic Weyl group invariant 
solution of the spectral problem of the Macdonald
$q$-difference operators. The $c$-function expansion 
is its explicit expansion in
terms of the basis of the space of meromorphic solutions 
of the spectral problem
consisting of $q$-analogs of the Harish-Chandra series. 
We express the expansion coefficients in terms of
a $q$-analog of the Harish-Chandra $c$-function, which is explicitly
given as product of $q$-Gamma functions. The
$c$-function expansion shows that the basic hypergeometric function 
formally is a $q$-analog of the Heckman-Opdam hypergeometric function,
which in turn specializes to elementary
spherical functions on noncompact Riemannian symmetric spaces for
special values of the parameters. 
\end{abstract}
\maketitle
\setcounter{tocdepth}{1}
\tableofcontents
\section{Introduction}
In this paper we establish the $c$-function expansion 
of a basic hypergeometric function $\mathcal{E}_+$ associated to 
root systems. Besides the base $q$, 
the basic hypergeometric function $\mathcal{E}_+$
depends on
a choice of a multiplicity function $k$
on an affine root system naturally associated to the underlying
based root system data. 
It will become apparent from the $c$-function expansion
that $\mathcal{E}_+$ formally is a $q$-analog of 
the Heckman-Opdam \cite{HO,HS,O} hypergeometric function, which in turn
reduces to the elementary
spherical functions on noncompact Riemannian symmetric spaces
for special parameter values.
We distinguish three important subclasses of the theory:
the reduced case, the $\textup{GL}_m$ case
and the nonreduced case.

In the reduced case $\mathcal{E}_+$ 
is Cherednik's global spherical function from \cite{CM,CWhit,CO},
or a reductive extension thereof. It
is a Weyl group invariant, meromorphic, selfdual common eigenfunction of the
Macdonald $q$-difference operators, constructed as an explicit
convergent series in symmetric Macdonald polynomials.
In the rank one case $\mathcal{E}_+$ can be explicitly related to the
basic hypergeometric series solutions of Heine's basic hypergeometric 
$q$-difference equation 
(see Subsection \ref{GLsection}). 

The $\textup{GL}_m$ case is a special case of the reduced case with
the underlying root system of type $\textup{A}_{m-1}$. 
It is of special interest since it
relates to Ruijsenaars' \cite{R} relativistic quantum trigonometric
Calogero-Moser model. In fact, the associated
Macdonald $q$-difference operators were first written
down by Ruijsenaars \cite{R} as
the corresponding quantum Hamiltonians.
 
In the nonreduced case the associated affine root system 
is the nonreduced affine root system of type $C^\vee C_n$. 
The multiplicity function $k$ now comprises 
five degrees of freedoms (four if the rank $n$
is equal to one). 
The associated basic hypergeometric function
$\mathcal{E}_+$ was constructed in \cite{St2}. 
Duality of $\mathcal{E}_+$ involves now a nontrivial transformation
of the multiplicity $k$ to a dual multiplicity function $k^d$
(we use the convention that the dual multiplicity function $k^d$
equals $k$ in the reduced case).
The associated Macdonald $q$-difference operators
include Koornwinder's \cite{Ko} multivariable extension of the 
Askey-Wilson \cite{AW}
second-order $q$-difference operator. 
It is the nonreduced case which is expected to be amenable to
generalizations to the elliptic level, cf. \cite{Ra}.

The basic Harish-Chandra series $\widehat{\Phi}_\eta(\cdot,\gamma)$ 
with base point given by a torus element $\eta$ is a meromorphic
common eigenfunction of the Macdonald $q$-difference operators having
a converging series expansion of the form
\[
\widehat{\Phi}_\eta(t,\gamma)=\widehat{\mathcal{W}}_\eta(t,\gamma)
\sum_{\mu\in Q_+}
\Gamma_\mu(\gamma)t^{-\mu},\qquad \Gamma_0(\gamma)=1
\]
deep in the appropriate asymptotic sector,
where $Q_+$ consists of 
the elements in the root lattice that can be written as
sum of positive roots. 
The prefactor 
$\widehat{\mathcal{W}}_\eta(t,\gamma)$ is an explicit 
quotient of theta functions
satisfying the asymptotic Macdonald $q$-difference equations
(see Subsection \ref{bHCs}). It is normalized such that it 
reduces to the natural choice \eqref{good}
of the prefactor when restricting $t$ to the $q$-lattice containing
$\eta\gamma_{0,d}$, where
$\gamma_{0,d}^{-1}$ (respectively $\gamma_0^{-1}$) denotes the torus element
associated to the $k^d$-deformation (respectively $k$-deformation)
of the half sum of positive roots, cf. \eqref{gamma0}. 
For the construction of the basic Harish-Chandra
series $\widehat{\Phi}_\eta$ we follow closely \cite{vMSt,vM}.

Let $W_0$ be the Weyl group of the underlying finite root system.
The space of common meromorphic eigenfunctions of the 
Macdonald $q$-difference operators 
has, for generic $\gamma$,
the $W_0$-translates $\widehat{\Phi}_\eta(\cdot,w\gamma)$
($w\in W_0$) as a linear basis over the field of quasiconstants
(this follows from combining and extending 
\cite[Cor. 5.14]{vMSt}, \cite[Rem. 5.13]{vM} 
and \cite[Thm. 5.16]{StKZ}). Hence, for generic $\gamma$,
we have
\begin{equation}\label{cfe}
\mathcal{E}_+(t,\gamma)=\widehat{c}_\eta(\gamma_0)^{-1}
\sum_{w\in W_0}\widehat{c}_\eta(w\gamma)
\widehat{\Phi}_\eta(t,w\gamma)
\end{equation}
for a unique coefficient $\widehat{c}_\eta(\gamma)$, which turns out to be
independent of $t$ due to the particular choice
$\widehat{\mathcal{W}}_\eta$ of the prefactor.
We will call \eqref{cfe} the (monic form of) the $c$-function expansion of
$\mathcal{E}_+$. We will prove
the following explicit expression 
\begin{equation}\label{cintro}
\widehat{c}_\eta(\gamma)=\frac{\vartheta((w_0\eta)^{-1}\xi\gamma)}
{\vartheta(\xi\gamma)}c_{k^d,q}(\gamma)
\end{equation}
for the expansion coefficient,
where $w_0\in W_0$ is the longest Weyl group element, 
$\vartheta(\cdot)$ is the theta-function \eqref{vartheta}
associated to the given root system data, $\xi$ is an explicit
torus element depending on the multiplicity function $k$
(see Corollary \ref{monicc} for the explicit expression of $\xi$; 
in the reduced case it is the unit element $1$ of the complex torus)
and $c_{k^d,q}(\gamma)$ \eqref{expl2} is ``half'' 
of the inverse of the dual weight function of the associated symmetric 
Macdonald-Koornwinder polynomials. The expression of
$c_{k,q}(\cdot)$ as product of $q$-shifted factorials 
(equivalently, as product of $q$-Gamma functions) is given by 
\eqref{Vr} in the reduced case and by \eqref{Vnr} in the nonreduced case. 
It is the $q$-analog 
of the Gindikin-Karpelevic \cite{GK} type product formula 
\cite[Def. 6.4]{HO}
of the Harish-Chandra $c$-function for
the Heckman-Opdam hypergeometric function. 

Note that for $\eta=1$ the theta function factors 
in the expression for $\widehat{c}_\eta(\gamma)$ cancel out.
The $c$-function expansion \eqref{cfe} thus simplifies to
\begin{equation}\label{cfe1}
\mathcal{E}_+(t,\gamma)=c_{k^d,q}(\gamma_0)^{-1}\sum_{w\in W_0}
c_{k^d,q}(w\gamma)\widehat{\Phi}_1(t,w\gamma).
\end{equation}
Comparing this formula for $t$ on the $q$-lattice containing $\gamma_{0,d}$
to the $c$-function expansion \cite[Part I, Def. 4.4.1]{HS}
of the Heckman-Opdam hypergeometric function,
it is apparent that $\mathcal{E}_+$ is formally a $q$-analog of
the Heckman-Opdam hypergeometric function.
The corresponding classical limit 
$q\rightarrow 1$ can be made rigorous if 
the underlying finite root system is of type $\textup{A}_1$,
see \cite{Klim}. In this paper we will not touch upon
making the limit rigorous in general, see 
\cite[Thm. 4.5]{CWhit} for further results in this direction.

It is important to consider the $c$-function expansion for
arbitrary $\eta$. In the rank one
nonreduced case, a selfdual
Fourier transform with Fourier kernel $\mathcal{E}_+$
and (Plancherel) density 
\[
\mu_\eta(\gamma)=\frac{1}{\widehat{c}_\eta(\gamma)\widehat{c}_\eta(\gamma^{-1})}
\]
was defined and studied in \cite{KSsu11,KS}. The  extra 
theta function contributions  in $\mu_\eta(\gamma)$
compared to the
usual weight function $\mu_1(\gamma)$ of the 
Macdonald-Koornwinder polynomials (which, in the present
nonreduced rank one setup,
are the Askey-Wilson \cite{AW} polynomials)
give rise to an infinite sequence of discrete mass points in the
associated (Plancherel) measure. In the interpretation as the inverse of
a spherical Fourier transform on the quantum
$\textup{SU}(1,1)$ group these mass points account for the
contributions of the strange series representations of the 
quantized universal enveloping algebra (see \cite{KSsu11}).

The basic hypergeometric function $\mathcal{E}_+(t,\gamma)$ is selfdual,
\[
\mathcal{E}_+(t,\gamma)=\mathcal{E}_{+,d}(\gamma^{-1},t^{-1}),
\]
where $\mathcal{E}_{+,d}$ is the basic hypergeometric function with
respect to the dual $k^d$ of the multiplicity function $k$.
This implies that $\mathcal{E}_+(t,\gamma)$ solves a 
bispectral problem,
in which dual Macdonald $q$-difference equations acting on $\gamma$
are added to the original Macdonald $q$-difference equations acting on $t$.
We show that a suitable, explicit renormalization $\Phi(\cdot,\cdot)
=\Phi(\cdot,\cdot;k,q)$ of the 
basic Harish-Chandra series $\widehat{\Phi}_\eta$ also becomes a selfdual
solution of the bispectral problem. We will derive 
the $c$-function expansion \eqref{cfe} as a consequence
of the more refined asymptotic expansion of $\mathcal{E}_+$,
\begin{equation}\label{cfesd}
\mathcal{E}_+(t,\gamma)=\sum_{w\in W_0}\mathfrak{c}(t,w\gamma)\Phi(t,w\gamma),
\end{equation}
where $\mathfrak{c}(t,\gamma)$ now is an explicit 
meromorphic function, quasiconstant
in both $t$ and $\gamma$. 

To prove the existence of an expansion of the form \eqref{cfesd}
we make essential use of Cherednik's \cite{C} double affine Hecke algebra
and of the bispectral quantum 
Knizhnik-Zamolodchikov (KZ) equations from \cite{vMSt,vM}.
We show that $\mathcal{E}_+$ is the Hecke algebra 
symmetrization of a nonsymmetric
analog $\mathcal{E}$ of the basic hypergeometric function, whose
fundamental property is an explicit transformation rule relating the
action of the double affine Hecke algebra 
on the first torus variable to the action of the double
affine Hecke algebra on the second torus variable
(this goes back to 
\cite{CM} in the reduced case and \cite{St2} in the nonreduced case). 
The Hecke algebra symmetrizer acting on such functions
factorizes as 
$\phi\circ\psi$ with $\psi$ mapping into
the space $\mathcal{K}^{W_0\times W_0}$
of Weyl group invariant meromorphic solutions of the bispectral
quantum KZ equations. The map $\phi$ is the difference
Cherednik-Matsuo map from
\cite{CKZ2}. This implies that 
the basic hypergeometric function $\mathcal{E}_+$
is the image under $\phi$ of the 
Weyl group invariant meromorphic solution $\psi(\mathcal{E})\in
\mathcal{K}^{W_0\times W_0}$
of the bispectral quantum KZ equations. This observation is essential
because it allows us to use 
the asymptotic analysis of the bispectral quantum KZ equations
from \cite{vMSt,vM}.
It implies that the space $\mathcal{K}$ 
of meromorphic solutions of the bispectral quantum KZ equations
has a basis over the field of quasiconstants defined in terms
of $W_0$-translates of
a selfdual asymptotically free solution $F$. The image of $F$ under
the difference Cherednik-Matsuo map $\phi$ is the selfdual basic
Harish-Chandra series $\Phi$ in \eqref{cfesd}.

In Theorem \ref{mainresult} we give an explicit expression of
the quasiconstant coefficient $\mathfrak{c}(t,\gamma)$ in the expansion
\eqref{cfesd} as product of theta functions. The higher rank theta
function $\vartheta(\cdot)$ \eqref{vartheta} and Jacobi's one-variable
theta function \eqref{Jacobi} are both involved.
The coefficient $\mathfrak{c}(t,\gamma)$ splits in two factors,
the first factor is an explicit product of higher rank theta functions, 
the second factor is $c_{k^d,q}(\gamma)
\mathcal{S}_{k^d,q}(\gamma)/
\mathcal{L}_q(\gamma)$ with
$\mathcal{L}_q(\gamma)$
the leading term of the asymptotic series of $\Phi(\cdot,\gamma)$
and $\mathcal{S}_{k^d,q}(\gamma)$ the holomorphic function
capturing the singularities of
$\Phi(t,\gamma)$ in $\gamma$ (see Theorem \ref{properties} and
Definition \ref{bHCdef}). The appearance of the higher rank theta functions
and of $c_{k^d,q}(\gamma)$ is due to the asymptotics of a suitable 
renormalization of the basic hypergeometric function 
$\mathcal{E}_+(\cdot,\gamma)$, see Corollary \ref{decompcor} and 
Proposition \ref{poll} (in the reduced case the asymptotics of the basic
hypergeometric function was considered by Cherednik \cite[\S 4.2]{CWhit}).
Similarly to $c_{k^d,q}(\gamma)$, the factor
$\mathcal{S}_{k^d,q}(\gamma)/\mathcal{L}_q(\gamma)$ can be explicitly
expressed as product of $q$-Gamma functions. By the Jacobi triple
product identity their product $c_{k^d,q}(\gamma)
\mathcal{S}_{k^d,q}(\gamma)/\mathcal{L}_q(\gamma)$ admits an expression as
product of Jacobi theta functions.

Recently \cite{StqHC}
explicit connection coefficient formulas for the selfdual basic
Harish-Chandra series $\Phi$ are derived.
They do not lead to a new proof of the $c$-function expansion though, 
see \cite[\S 1.5]{StqHC} for a detailed discussion.

As an application of the $c$-function expansion we
establish pointwise asymptotics of the Macdonald-Koornwinder polynomials
in Subsection \ref{abcd} 
(the $L^2$-asymptotics was obtained by different methods in 
\cite{RuL, vD1, vD2}). In addition we relate and compare
in Subsections \ref{specialnr} and  \ref{GLsection}
our results to the classical theory of basic hypergeometric series \cite{GR}
when the rank of the underlying root system is one.

\section{The basic hypergeometric function}
In this section we give the definition of the basic hypergeometric function
associated to root systems. It was introduced by Cherednik in \cite{CM} 
for irreducible reduced twisted affine root systems. In \cite{St2} 
it was defined for
the nonreduced case (sometimes called the Koornwinder case, 
or $C^\vee C$ case). We give a uniform treatment in which we allow
extra freedom in the choice of
the associated translation lattice. This enables us to
include the
$\textup{GL}_m$-extension of the reduced type $A$ case in our treatment.

\subsection{Affine root systems and extended affine Weyl groups}\label{start}
In this subsection we recall well known facts on affine root systems and
affine Weyl groups (for further details see, e.g., \cite{C,M}).
Let $V$ be an Euclidean space of dimension $m$
with scalar product $\bigl(\cdot,\cdot\bigr)$
and corresponding norm $|\cdot|$.
Let $R_0\subset V$ be a finite set of nonzero vectors and let $V_0$ be its
real span. We suppose that $R_0\subset V_0$ is 
a crystallographic, reduced irreducible
root system. We write $R_{0,s}$ (respectively $R_{0,l}$) for the
subset of $R_0$ of short (respectively long) roots. If all roots of
$R_0$ have the same root length then $R_{0,s}=R_0=R_{0,l}$ by convention.
Let $n=\textup{dim}(V_0)\leq m$ be
the rank of $R_0$.
Let $\Delta_0=
(\alpha_1,\ldots,\alpha_n)$ be an ordered basis of $R_0$ and
$R_0=R_0^+\cup R_0^-$ the corresponding decomposition of $R_0$
in positive and negative roots. We order the basis elements in such
a way that $\alpha_n$ is a short root. 
We write $\varphi\in R_0^+$ (respectively
$\theta\in R_0^+$) for the corresponding 
highest root (respectively highest short
root). They coincide if $R_0$ has only one root length.
Let $Q=\bigoplus_{i=1}^n\mathbb{Z}\alpha_i$ be the root lattice and 
set $Q_+=\bigoplus_{i=1}^n\mathbb{Z}_{\geq 0}\alpha_i$.

View $\widehat{V}=V\oplus\mathbb{R}c$ as the space of
real valued affine linear functions on $V$ by
\[
v+rc: v^\prime\mapsto
(v,v^\prime)+r\qquad  (v,v^\prime\in V, r\in\mathbb{R}).
\]
We extend the scalar product $\bigl(\cdot,\cdot\bigr)$
to a semi-positive definite
form on $\widehat{V}$ such that the constant functions $\mathbb{R}c$
are in the radical.
The canonical action of the 
affine linear group $\textup{GL}_{\mathbb{R}}(V)\ltimes V$ on $V$
gives rise to a linear action on $\widehat{V}$
by transposition. We denote the resulting translation
actions by $\tau$. Thus $\tau(v)v^\prime=v+v^\prime$ and
\[
\tau(v)(v^\prime+rc)=v^\prime+(r-(v,v^\prime))c.
\]
For $0\not=\alpha\in V$ and $r\in\mathbb{R}$ let $s_{\alpha+rc}$
be the orthogonal reflection in the affine hyperplane 
$\{v\in V\,\, | \,\, (\alpha,v)=-r\}$. Then $s_{\alpha+rc}\in
\textup{GL}_{\mathbb{R}}(V)\ltimes V$. In fact, 
$s_{\alpha+rc}=\tau(-r\alpha^\vee)s_\alpha$ with
$\alpha^\vee:=2\alpha/|\alpha|^2$. 

The twisted reduced affine root system $R^\bullet$ associated to $R_0$ is 
\[
R^\bullet:=\{\alpha+r\frac{|\alpha|^2}{2}c \,\, | \,\, \alpha\in R_0,\,\,
r\in\mathbb{Z} \}\subset \widehat{V}.
\]
The affine Weyl group $W^\bullet$ of $R^\bullet$ is 
the subgroup of $\textup{GL}_{\mathbb{R}}(V)\ltimes V$
generated by $s_a$ ($a\in R^\bullet$). 
It preserves $R^\bullet$. In addition, $W^\bullet\simeq W_0\ltimes Q$
with $W_0$ the Weyl group of $R_0$. 
We extend the ordered basis $\Delta_0$ of
$R_0$ to an ordered basis
\[
\Delta=(a_0,a_1,\ldots,a_n):=\bigl(\frac{|\theta|^2}{2}c-\theta,\alpha_1,
\ldots,\alpha_n)
\]
of $R^\bullet$. It results in the decomposition
$R^\bullet=R^{\bullet,+}\cup R^{\bullet,-}$ of $R^\bullet$ 
in positive and negative roots.
 
The Weyl group $W_0$ and the
affine Weyl group $W^\bullet$ are Coxeter groups, with Coxeter generators
the simple reflections $s_i:=s_{\alpha_i}$ ($1\leq i\leq n$) respectively
$s_0:=s_{a_0},s_1,\ldots,s_n$. 

Let $Q^\vee$ be the coroot lattice of $R_0$, i.e. it is
the integral span of the 
coroots $\alpha^\vee=2\alpha/|\alpha|^2$ ($\alpha\in R_0$).
Fix a full lattice $\Lambda\subset V$ satisfying
$Q\subseteq\Lambda$ and
$\bigl(\Lambda,Q^\vee\bigr)\subseteq \mathbb{Z}$. 
\begin{rema}\label{startRemark}
In this remark we relate the 
triples $(R_0,\Delta_0,\Lambda)$ to the notion of a
based root datum (cf., e.g., \cite[\S 1]{Sp} for a survey on root
data). Using the notations from \cite[\S 1]{Sp}, suppose that
$\Psi_0=(X,\Phi,\Delta,X^\vee,\Phi^\vee,\Delta^\vee)$ is a nontoral based
root datum with associated perfect pairing
$\langle \cdot,\cdot\rangle: X\times X^\vee\rightarrow\mathbb{Z}$
and associated bijection $\alpha\mapsto\alpha^\vee$ of $\Phi$
onto $\Phi^\vee$.
Assume that the root system $\Phi$ is reduced and irreducible. Choose
a Weyl group invariant scalar product $\bigl(\cdot,\cdot\bigr)$
on $V:=\mathbb{R}\otimes_{\mathbb{Z}}X$. 
Now we use the scalar product to embed $X^\vee$ as a
full lattice in $V$. Thus $\xi\in X^\vee$, regarded as element of 
$V$, is characterized by the requirement that
$(\xi,x)$ equals $\langle x,\xi\rangle$ for
all $x\in X$. The element $\alpha^\vee\in\Phi^\vee$ then corresponds
to the coroot $2\alpha/|\alpha|^2$ in $V$. 
It follows that the triple $(R_0,\Delta_0,\Lambda):=(\Phi,\Delta,X)$ in $V$ 
satisfies the desired properties. 
\end{rema}
Let 
\[P:=\{\lambda\in V_0 \,\, | \,\,
(\lambda,\alpha^\vee)\in\mathbb{Z}\quad\forall\,\alpha\in R_0\}
\]
be the weight lattice of $R_0$. Let $\widetilde{\varpi}_i\in P$
($1\leq i\leq n$) be the fundamental weights of $P$ with respect to
the ordered basis $\Delta_0$ of $R_0$. In other words,
$\widetilde{\varpi}_i\in V_0$ is characterized by 
$(\widetilde{\varpi}_i,\alpha_j^\vee)=\delta_{i,j}$ (Kronecker delta
function) for $1\leq i,j\leq n$.
Since $Q\subseteq P$ with finite index and $Q\subseteq\Lambda$, there
exists for each $i\in\{1,\ldots,n\}$ a smallest natural number $m_i$
such that $\varpi_i:=m_i\widetilde{\varpi}_i\in\Lambda$. Then
$\{\varpi_i\}_{i=1}^n$ is a basis of a $W_0$-invariant, rank $n$
sublattice of $\Lambda\cap V_0$ with the basis elements satisfying 
$(\varpi_i,\alpha_j^\vee)=0$ if
$j\not=i$ and $(\varpi_i,\alpha_i^\vee)=m_i\in\mathbb{Z}_{>0}$.
Set
\[
\Lambda_c:=\Lambda\cap V_0^\perp.
\]
Note that $L\Lambda\subseteq
\Lambda_c\oplus\bigoplus_{i=1}^n\mathbb{Z}\varpi_i$ for
$L:=m_1\cdots m_n$. In particular, 
$\Lambda_c$ is a full sublattice of $V_0^\perp$.

We list here the three key examples
of triples $(R_0,\Delta_0,\Lambda)$.
\begin{eg}\label{exampleLambda}
{\bf (i)} If $V_0=V$ then a natural choice for $\Lambda$ is the
weight lattice $P$. Then $\varpi_i=\widetilde{\varpi}_i$.\\
{\bf (ii)} Let $V=\mathbb{R}^m$ with orthonormal basis $\{\epsilon_i\}_{i=1}^m$.
Take $R_0=\{\epsilon_i-\epsilon_j\}_{1\leq i\not=j\leq m}$ the root system
of type $\textup{A}_{m-1}$ with ordered basis 
\[
\Delta_0=(\epsilon_1-
\epsilon_2,\epsilon_2-\epsilon_3,\ldots,\epsilon_{m-1}-\epsilon_m).
\]
Then $V_0\subset V$ is of codimension one and $n=m-1$. 
In this case we can take
$\Lambda=\bigoplus_{i=1}^m\mathbb{Z}\epsilon_i$.
The correponding elements $\varpi_i\in\Lambda$ ($1\leq i<m$) are
given by $\varpi_i=m_i\widetilde{\varpi}_i$ with
$m_i$ the smallest natural number such that $im_i\in m\mathbb{Z}$
and with
\[
\widetilde{\varpi}_i=\epsilon_1+\cdots+\epsilon_i-
\frac{i}{m}(\epsilon_1+\cdots+\epsilon_m).
\]
The lattice $\Lambda_c$ is generated by $\epsilon_1+\cdots+\epsilon_m$.\\
{\bf (iii)} Let $R_0\subset V_0=V=\mathbb{R}^n$ be the root system of type
$\textup{A}_1$ if $n=1$ and of type $\textup{B}_n$ if $n\geq 2$. 
In this case we can take $\Lambda=Q$. The corresponding elements
$\varpi_i$ ($1\leq i\leq n$) are given by
$\varpi_i=\widetilde{\varpi}_i$
($1\leq i<n$) and $\varpi_n=2\widetilde{\varpi}_n$. They form a
basis of $\Lambda$.
\end{eg}

We call $W:=W_0\ltimes\Lambda$  
the extended affine Weyl group associated to the triple 
$(R_0,\Delta_0,\Lambda)$. It preserves $R^\bullet$ and it 
contains the affine Weyl group $W^\bullet$ as
normal subgroup. 

The length function on $W$ is defined by
\[
l(w)=\#\bigl(R^{\bullet,+}\cap w^{-1}R^{\bullet,-}\bigr),\qquad w\in W.
\]
Let $\Omega\subset W$ be the subgroup
\[
\Omega=\{ w\in W \,\, | \,\, l(w)=0\}.
\]
It normalizes $W^\bullet$, and $W\simeq \Omega\ltimes W^\bullet$. In particular,
$\Omega\simeq W/W^\bullet\simeq\Lambda/Q$ as abelian groups.

The action of $\Omega$ on $R^\bullet$ restricts to an action on the
unordered basis $\{a_0,\ldots,a_n\}$ of $R^\bullet$. We also view it as 
action on the indexing set $\{0,\ldots,n\}$ of the basis, so that
$ws_iw^{-1}=s_{w(i)}$ for $w\in\Omega$ and $0\leq i\leq n$. By the same
formula, $\Omega$ acts on the affine braid group $\mathcal{B}$
associated to the Coxeter system $(W^\bullet,(s_0,\ldots,s_n))$ by group
automorphisms.

The elements of the group $\Omega$ can alternatively be described as follows.
For $\lambda\in\Lambda$ write $u(\lambda)\in W$ for the
unique element in the coset $W_0\tau(\lambda)\subset W$ of minimal length.
We have $u(\lambda)=\tau(\lambda)$ if $\lambda\in\Lambda^-$,
where 
\[
\Lambda^{\pm}:=\{\lambda\in\Lambda \,\, | \,\, \pm(\lambda,\alpha^\vee)
\geq 0\quad \forall\,\alpha\in R_0^+\}
\]
is the set of dominant and antidominant weights in $\Lambda$ respectively.
Let 
\[
\Lambda_{min}^+:=\{\lambda\in\Lambda \,\, | \,\, 0\leq
(\lambda,\alpha^\vee)\leq 1\quad \forall\,\alpha\in R_0^+\}
\]
be the set of miniscule dominant weights in $\Lambda$.
Then $\Omega=\{u(\lambda) \,|\, \lambda\in\Lambda_{min}^+\}$.

Note that $(\Lambda,a^\vee)=\mathbb{Z}$
or $=2\mathbb{Z}$ for $a\in R^\bullet$, where $a^\vee=2a/(a,a)$.
Define a subset $S=S(R_0,\Delta_0,\Lambda)$ of the index set $\{0,\ldots,n\}$
of the simple affine roots by 
\[
S:=\{i\in\{0,\ldots,n\}\,\, | \,\, (\Lambda,a_i^\vee)=2\mathbb{Z}\}.
\]
Case by case verification shows that
$S=\emptyset$ or $\#S=2$. If $\#S=2$ and $n=1$ then $R_0$ is of type
$\textup{A}_1$. If $\#S=2$ and $n\geq 2$ then $R_0$ is of type
$\textup{B}_n$ and $S=\{0,n\}$ (recall that $\alpha_n$ is short).
We call $S=\emptyset$ the {\it reduced case}
and $\#S=2$ the {\it nonreduced case}. Thus the $\textup{GL}_m$ case 
(corresponding to Example \ref{exampleLambda}{\bf (ii)}) will
be regarded as a special case of the reduced case.

We define a $\Lambda$-dependent extension of the reduced
irreducible affine root system $R^\bullet$ as follows.
\begin{defi}
The irreducible affine root system $R=R(R_0,\Delta_0,\Lambda)\subset\widehat{V}$
is defined by
\[
R:=R^\bullet\cup\bigcup_{j\in S}W^\bullet(2a_j).
\]
\end{defi}

In the reduced case we simply have $R=R^\bullet$.
In this case the $W$-orbits of $R$ are in one to one
correspondence with the $W_0$-orbits of $R_0$. Concretely, the affine root
$\alpha+r\frac{|\alpha|^2}{2}c\in R$ lies in the same $W$-orbit
as $\beta+r^\prime\frac{|\beta|^2}{2}c\in R$ iff $\alpha\in W_0\beta$.

In the nonreduced case 
we have
\[
R=R^\bullet\cup W(2a_0)\cup W(2a_n).
\] 
It is the nonreduced irreducible affine root system of type 
$C^\vee C_n$, cf. \cite{M0}.
The basis 
$\Delta$ of $R^\bullet$ is also a basis of $R$ and $W$ is still
the associated affine Weyl group.
Note that $R$ now has five $W$-orbits
\[
W(a_0), W(2a_0), W(\varphi), W(\theta), W(2\theta).
\]
In the nonreduced case $\Lambda_{min}^+=\Lambda_c$ and  
$\Lambda=Q\oplus\Lambda_c$, hence $W=\Lambda_c\times W^\bullet$. 
The reductive extension $\Lambda_c$ of the root lattice
will always play a trivial role in the nonreduced case.
To simplify the presentation we will therefore assume in the remainder
of the paper that $V_0=V$, in particular
$\Lambda=Q$, $\Lambda_c=\{0\}$ and $\Omega=\{1\}$, if we are dealing
with the nonreduced case.

\subsection{The double affine Hecke algebra}
References for this subsection are \cite{C,M,StBook}.
We call a function $k: R\rightarrow\mathbb{C}^*$,
denoted by $a\mapsto k_a$, a multiplicity function if $k_{wa}=k_a$
for $w\in W$ and $a\in R$. We will assume throughout the paper that
\begin{equation}\label{condition1}
0<k_a<1\quad \forall a\in R.
\end{equation}
We write $k^\bullet$ for its restriction to
$R^\bullet$ and $k_i:=k_{a_i}^\bullet$ for $0\leq i\leq n$. We
set $k_{2a}:=k_a$ if $a\in R$ and $2a\not\in R$.
\begin{defi}
{\bf (i)} The affine Hecke algebra $H^\bullet(k^\bullet)=
H^\bullet(R_0,\Delta_0;k^\bullet)$ is the 
unique associative unital algebra over $\mathbb{C}$
with generators $T_0,\ldots,T_n$
satisfying the affine braid relations of $\mathcal{B}$
and satisfying the quadratic relations
\[(T_i-k_i)(T_i+k_i^{-1})=0,\qquad 0\leq i\leq n.
\]
{\bf (ii)} The extended affine Hecke algebra $H(k^\bullet)=H(R_0,\Delta_0;
k^\bullet)$ is the crossed product algebra
$\Omega\ltimes H^\bullet(k^\bullet)$, where $\Omega$ acts by algebra automorphisms
on $H^\bullet(k^\bullet)$ by $w(T_i)=T_{w(i)}$ for $w\in\Omega$ and 
$0\leq i\leq n$.
\end{defi}
Recall that the lattice $\Lambda$ is $W_0$-stable, hence $W_0$ acts
on the complex algebraic torus $T=\textup{Hom}(\Lambda,\mathbb{C}^*)$
by transposition. Writing $t^\lambda$ for the value of $t\in T$ at
$\lambda\in\Lambda$, we thus have $(w^{-1}t)^\lambda=t^{w\lambda}$
for $w\in W_0$.

Fix $0<q<1$.
The $W_0$-action
on $T$ extends to a $q$-dependent left $W$-action $(w,t)\mapsto w_qt$
on $T$ by
\[
\tau(\lambda)_qt:=q^\lambda t,
\]  
where $q^\lambda\in T$ is defined by $\nu\mapsto q^{(\lambda,\nu)}$.

Let $\mathbb{C}[T]$ be the space of regular functions on $T$ with 
$\mathbb{C}$-basis
the monomials $t\mapsto t^\lambda$ ($\lambda\in\Lambda$). 
Let $\mathbb{C}(T)$ be the corresponding quotient field. Let
$\mathcal{M}(T)$ be the field of meromorphic functions on $T$. By transposition
the $q$-dependent $W$-action on $T$ gives an action by
field automorphisms on both $\mathbb{C}(T)$ and $\mathcal{M}(T)$. This action
will also be denoted by $(w,p)\mapsto w_qp$.
Let $\mathbb{C}(T)\rtimes_qW\subset\mathcal{M}(T)\rtimes_qW$
be the corresponding crossed product algebras. They canonically act
on $\mathcal{M}(T)$ by $q$-difference reflection operators.

Write for $t\in T$ and $a=\alpha+r\frac{|\alpha|^2}{2}c\in R^\bullet$,
\[
t_q^a:=q_\alpha^rt^\alpha,
\]
where $q_\alpha:=q^{\frac{|\alpha|^2}{2}}$.
Define for $a\in R^\bullet$ the rational function $c_a=c_a(\cdot;k,q)
\in\mathbb{C}(T)$ by
\begin{equation}\label{ca}
c_a(t):=\frac{(1-k_ak_{2a}t_q^a)(1+k_ak_{2a}^{-1}t_q^a)}{(1-t_q^{a})(1+t_q^a)}.
\end{equation}
It satisfies $c_a(w_q^{-1}t)=c_{wa}(t)$
for $a\in R^\bullet$ and $w\in W$. The following fundamental result 
is due to Cherednik in the reduced case
(see \cite[Thm. 3.2.1]{C} and references therein)
and due to Noumi \cite{N} in the nonreduced case.
\begin{thm}
There exists a unique faithful algebra homomorphism
\[
\pi=\pi_{k,q}: H(k^\bullet)\rightarrow \mathbb{C}(T)\rtimes_qW
\]
satisfying
\begin{equation*}
\begin{split}
\pi_{k,q}(T_i)&=k_i+k_i^{-1}c_{a_i}(s_{i,q}-1),\qquad 0\leq i\leq n,\\
\pi_{k,q}(w)&=w_q,\qquad\qquad\qquad\qquad\quad w\in\Omega.
\end{split}
\end{equation*}
\end{thm}
\begin{rema}
In the reduced case $\pi_{k,q}$ is a one parameter family of algebra embeddings
of $H(k^\bullet)$ (with $q$ being the free parameter).
In the nonreduced case $\pi_{k,q}$ is a three parameter
family of algebra embeddings of 
$H(k^\bullet)$ (with $q,k_{2\theta},k_{2a_0}$ being
the free parameters).
\end{rema}
The double affine Hecke algebra $\mathbb{H}=\mathbb{H}(k,q)$ is the
subalgebra of $\mathbb{C}(T)\rtimes_qW$ generated by $\mathbb{C}[T]$
and $\pi_{k,q}(H(k^\bullet))$. In the remainder of the paper we will 
often identify $H(k^\bullet)$ with its $\pi_{k,q}$-image
in $\mathbb{H}(k,q)$. 
In addition we write for $\lambda+rc$ ($\lambda\in\Lambda$
and $r\in\mathbb{R}$), 
\[
X^{\lambda+rc}_q=q^{r}X^\lambda
\]
for the element in the double affine Hecke algebra corresponding to
the regular function $t\mapsto q^{r}t^\lambda$ on $T$.

Under the canonical action of $\mathbb{C}(T)\rtimes_qW$ on $\mathbb{C}(T)$,
the subspace $\mathbb{C}[T]$ is $\mathbb{H}$-stable. It is called the
basic, or polynomial, representation of the double affine Hecke algebra.

\subsection{Nonsymmetric Macdonald-Koornwinder polynomials}
The results on nonsymmetric Macdonald and Koornwinder polynomials
in this subsection are well known. In the reduced case
they are due to Cherednik (the definition of the nonsymmetric Macdonald
polynomial was independently given by Macdonald), see, e.g., \cite[\S 3.3]{C}
and \cite{M} and references therein. In the nonreduced case the results
in this subsection are from \cite{N,Sa,St}. The current uniform
presentation of these results follows \cite{StBook}.

For $w\in W$ with reduced expression
$w=u(\lambda)s_{i_1}\ldots s_{i_l}$ 
($\lambda\in\Lambda_{min}^+$, $0\leq i_j\leq n$ and $l=l(w)$) we write
\[
T_w:=u(\lambda)T_{i_1}\cdots T_{i_l}\in H(k^\bullet).
\] 
The expression is independent of the choice of reduced expression.
By unpublished results of Bernstein and Zelevinsky (cf. \cite{Lu}),
there exists a unique injective algebra homomorphism 
$\mathbb{C}[T]\hookrightarrow H(k^\bullet)$, which we denote by 
$p\mapsto p(Y)$, such that $Y^\lambda=T_{\tau(\lambda)}$ 
for $\lambda\in\Lambda^+$. Its image in $H(k^\bullet)$ is denoted by 
$\mathbb{C}_Y[T]$. The center $Z(H(k^\bullet))$
of $H(k^\bullet)$ is $\mathbb{C}_Y[T]^{W_0}$. 

For $x\in\mathbb{C}^*$ and $\alpha\in R_0$
define $x^{\alpha^\vee}\in T$ by $\lambda\mapsto x^{(\lambda,\alpha^\vee)}$
($\lambda\in\Lambda$). 
Let $\gamma_0=\gamma_0(k)\in T$ be the torus element
\begin{equation}\label{gamma0}
\gamma_0:=\prod_{\alpha\in R_0^+}\bigl(k_\alpha^{\frac{1}{2}}
k_{\alpha+|\alpha|^2c/2}^{\frac{1}{2}}\bigr)^{-\alpha^\vee}\in T.
\end{equation}
More generally, define for $\lambda\in\Lambda$ the element
$\gamma_\lambda:=\gamma_\lambda(k,q)\in T$ by
\[
\gamma_\lambda:=u(\lambda)_q\gamma_0.
\]
For $\lambda\in\Lambda^-$ we thus have 
$\gamma_\lambda=q^\lambda\gamma_0$. 
\begin{thm}\label{defNSM}
Let $\lambda\in\Lambda$. There exists a unique $P_\lambda=P_\lambda(\cdot;
k,q)\in\mathbb{C}[T]$ such that 
\[\pi_{k,q}(p(Y))P_\lambda=p(\gamma_\lambda^{-1})P_\lambda
\qquad \forall\, p\in\mathbb{C}[T]
\]
and such that the coefficient of $t^\lambda$
in the expansion of $P_\lambda(t)$ in monomials $t^\nu$ ($\nu\in\Lambda$)
is one. 
\end{thm}
$P_\lambda$ is the monic nonsymmetric Macdonald polynomial
of degree $\lambda$ in the reduced case and the nonsymmetric monic
Koornwinder polynomial in the nonreduced case. We refer to $P_\lambda$
in the remainder of the text as the monic 
nonsymmetric Macdonald-Koornwinder polynomial (similar terminology
will be used later
for the normalized and symmetrized versions of $P_\lambda$).

Write $k^{-1}$ for the multiplicity function $a\mapsto k_a^{-1}$. Similarly
to Theorem \ref{defNSM} there exists, for $\lambda\in\Lambda$, a unique
$P_\lambda^\prime=P_\lambda^\prime(\cdot;k,q)\in\mathbb{C}[T]$ such that
\[
\pi_{k^{-1},q^{-1}}(p(Y))P_\lambda^\prime=p(\gamma_\lambda)P_\lambda^\prime\qquad
\forall\, p\in\mathbb{C}[T]
\]
and such that the coefficient of $t^\lambda$ in the expansion of 
$P_\lambda^\prime(t)$ in monomials $t^\nu$ ($\nu\in\Lambda$) is one.

Next we define the normalized versions of $P_\lambda$ and $P_\lambda^\prime$.
For this we first need to recall the evaluation formulas for $P_\lambda$
and $P_\lambda^\prime$.

The multiplicity function $k^d$ on $R$ dual to $k$ is defined as follows.
In the reduced case $k^d:=k$. In the nonreduced case $k^d_0:=k_{2\theta}$,
$k^d_{2\theta}:=k_0$ and the values on the remaining $W$-orbits of $R$
are unchanged. Set $\gamma_{\lambda,d}=\gamma_\lambda(k^d,q)$.
The evaluation formulas for the nonsymmetric Macdonald-Koornwinder polynomials
then read
\begin{equation}\label{evalform}
\begin{split}
P_\lambda(\gamma_{0,d})&=\prod_{a\in R^{\bullet,+}\cap
u(\lambda)^{-1}R^{\bullet,-}}(k_a^d)^{-1}c_a(\gamma_0;k^d,q),\\
P_\lambda^\prime(\gamma_{0,d}^{-1})&=\prod_{a\in R^{\bullet,+}\cap
u(\lambda)^{-1}R^{\bullet,-}}k_a^dc_a(\gamma_0^{-1};(k^d)^{-1},q^{-1}).
\end{split}
\end{equation}
By the conditions on the parameters $k_a$ and $q$ we have
$P_\lambda(\gamma_{0,d})\not=0\not=
P_\lambda^\prime(\gamma_{0,d}^{-1})$ for all $\lambda\in\Lambda$. 
Hence the following definition makes sense.

\begin{defi}
Let $\lambda\in\Lambda$. The normalized nonsymmetric Macdonald-Koornwinder 
polynomial $E(\gamma_\lambda;\cdot)=
E(\gamma_\lambda;\cdot;k,q)\in\mathbb{C}[T]$ of degree $\lambda$ is defined by
\[
E(\gamma_\lambda;t):=\frac{P_\lambda(t)}{P_\lambda(\gamma_{0,d})}.
\]
\end{defi}
Similarly we define $E^\prime(\gamma_\lambda^{-1};\cdot)=
E^\prime(\gamma_\lambda^{-1};\cdot;k,q)\in\mathbb{C}[T]$ by 
\[
E^\prime(\gamma_\lambda^{-1};t):=
\frac{P_\lambda^\prime(t)}{P_\lambda^\prime(\gamma_{0,d}^{-1})}.
\]
It is related to $E(\gamma_\lambda;t)$ by the formula 
\begin{equation}\label{conjugationsymm}
E^\prime(\gamma_\lambda^{-1};t^{-1})=
k_{w_0}^{-1}\bigl(\pi_{k,q}(T_{w_0})E(\gamma_{-w_0\lambda};\cdot)\bigr)(t),
\end{equation}
where $w_0\in W_0$ is the longest Weyl group element and
$k_w:=\prod_{\alpha\in R_0^+\cap w^{-1}R_0^-}k_\alpha$ for $w\in W_0$
(see \cite[(3.3.26)]{C} for a proof of \eqref{conjugationsymm}
in the reduced case; its proof easily extends
to the nonreduced case).

An important property of the normalized 
nonsymmetric Macdonald-Koornwinder
polynomials is duality,
\[
E(\gamma_\lambda;\gamma_{\nu,d};k,q)=
E(\gamma_{\nu,d};\gamma_\lambda;k^d,q),\qquad
\forall\,\lambda,\nu\in\Lambda.
\]
A similar duality formula is valid for $E^\prime$.

Define $N(\lambda)=N(\lambda;k,q)$ ($\lambda\in\Lambda$) by
\[
N(\lambda):=\prod_{a\in R^{\bullet,+}\cap u(\lambda)^{-1}R^{\bullet,-}}
\frac{c_{-a}(\gamma_0;k^d,q)}{c_a(\gamma_0;k^d,q)}.
\]
They appear as the quadratic norms of the nonsymmetric Macdonald-Koornwinder
polynomials. 
Concretely, if the parameters satisfy the additional conditions
$|k_ak_{2a}^{-1}|\leq 1$ for all $a\in R$ (this only gives additional
constraints in the nonreduced case), then 
\[
\langle E(\gamma_\lambda;\cdot;k,q),
E(\gamma_\nu^{-1};k^{-1},q^{-1})\rangle_{k,q}=\langle 1,1\rangle_{k,q}
N(\lambda;k,q)\delta_{\lambda,\nu}
\]
for all $\lambda,\nu\in\Lambda$ with respect to the
sesquilinear pairing 
\[
\langle p_1,p_2\rangle_{k,q}:=\int_{T_u}p_1(t)\overline{p_2(t)}
\Bigl(\prod_{a\in R^{\bullet,+}}\frac{1}{c_a(t;k,q)}\Bigr)dt,
\qquad p_1,p_2\in\mathbb{C}[T].
\]
Here
$T_u:=\textup{Hom}(\Lambda,S^1)\subset T$ with $S^1$ the unit circle
in the complex plane, and 
$dt$ is the normalized Haar measure on the compact torus $T_u$. 

\subsection{Theta functions}
The results in this section are 
from \cite[\S 3.2]{C} in the reduced case
and from \cite{St2} in the nonreduced case. 

The $q$-shifted factorial is 
\[\bigl(x;q\bigr)_{r}:=\prod_{i=0}^{r-1}(1-q^ix),\qquad 
r\in\mathbb{Z}_{\geq 0}\cup\{\infty\}
\]
(by convention empty products are equal to one). 
The $q$-Gamma function is 
\begin{equation}\label{qgamma}
\Gamma_q(x):=
(1-q)^{1-x}\frac{\bigl(q;q\bigr)_{\infty}}{\bigl(q^x;q\bigr)_{\infty}},
\end{equation}
see \cite[\S 1.10]{GR}. Set
\begin{equation}\label{Jacobi}
\theta(x;q):=\bigl(q;q\bigr)_{\infty}\bigl(x;q\bigr)_{\infty}
\bigl(q/x;q\bigr)_{\infty}.
\end{equation}
It is the Jacobi theta function $\sum_{r=-\infty}^{\infty}
q^{\frac{r^2}{2}}(-q^{-\frac{1}{2}}x)^r$, written in multiplicative form
via the Jacobi triple product identity. It satisfies
the functional equations
\[
\theta(q^rx;q)=(-x)^{-r}q^{-\frac{r(r-1)}{2}}\theta(x;q)\qquad 
\forall\,r\in\mathbb{Z}.
\]
The theta function associated to the lattice $\Lambda$ is
the holomorphic $W_0$-invariant
function $\vartheta(\cdot)=\vartheta_\Lambda(\cdot)$ on $T$ defined by 
\begin{equation}\label{vartheta}
\vartheta(t):=\sum_{\lambda\in\Lambda}q^{\frac{|\lambda|^2}{2}}t^\lambda.
\end{equation}
Since the base for $\vartheta(\cdot)$ will always be $q$ we do not
specify it in the notation.
The theta function $\vartheta(\cdot)$
satisfies the functional equations $\vartheta(q^\lambda t)=
q^{-\frac{|\lambda|^2}{2}}t^{-\lambda}\vartheta(t)$ for all $\lambda\in\Lambda$.
\begin{rema}
We will always specify the variable dependence, 
$\theta(\cdot;q)$ and $\vartheta(\cdot)$,
to avoid confusion with the
highest short root $\theta$ and the highest root $\vartheta$ of $R_0$.
\end{rema}
\begin{defi}
Define $G=G_{k,q}\in\mathcal{M}(T)$ by
\[
G(t):=\vartheta(t)^{-1}
\]
in the reduced case and
\[
G(t):=\bigl(q_\theta^2;q_\theta^2\bigr)_{\infty}^{-n}
\prod_{\alpha\in R_{0,s}}\bigl(-q_\theta k_0k_{2a_0}^{-1}t^\alpha;
q_\theta^2\bigr)_{\infty}^{-1}
\]
in the nonreduced case.
\end{defi}
Note that $G(\cdot)$ is $W_0$-invariant, and that $G(t)=G(t^{-1})$.
\begin{rema}\label{thetarem}
In the nonreduced case the 
set $R_0^{s,+}$ of positive short roots
is an orthogonal basis of $V$ and a $\mathbb{Z}$-basis of 
$\Lambda=Q$ (cf. Subsection \ref{specialnr}). 
By the Jacobi triple product identity it then
follows that
$G(t)$ equals $\vartheta(t)^{-1}$ in the nonreduced case if $k_0=k_{2a_0}$.
\end{rema}

We recall the most fundamental property of $G(\cdot)$ in the following
proposition. It implies that $G(\cdot)$ serves as the analog of the Gaussian
in the context of the double affine Hecke algebra. For proofs and more facts 
we refer to \cite{CM,St2}.
\begin{prop}
{\bf (i)} Given a multiplicity function $k$ on $R$, the assignment
$k_{a_0}^\tau:=k_{2a_0}$, $k_{2a_0}^\tau=k_{a_0}$ and $k_{a_i}^\tau:=k_{a_i}$,
$k_{2a_i}^\tau:=k_{2a_i}$ for $1\leq i\leq n$ determines a
multiplicity function $k^\tau$ on $R$.\\
{\bf (ii)} There exists a unique algebra isomorphism
$\tau: \mathbb{H}(k,q)\overset{\sim}{\longrightarrow}
\mathbb{H}(k^\tau,q)$ satisfying
\begin{equation*}
\begin{split}
\tau(T_0)&=X_q^{-a_0}T_0^{-1},\\
\tau(T_i)&=T_i,\qquad\qquad\qquad\qquad\,\, 1\leq i\leq n,\\
\tau(X^\lambda)&=X^\lambda,\qquad\quad\qquad\qquad\,\,\,\,\, \lambda\in\Lambda,\\
\tau(u(\lambda))&=q^{-\frac{|\lambda|^2}{2}}X^\lambda u(\lambda),\quad\qquad\,\,\,
\lambda\in\Lambda_{min}^+.
\end{split}
\end{equation*}
{\bf (iii)} For all $Z\in\mathbb{H}(k,q)$ we have
\[
G_{k,q}(\cdot)ZG_{k,q}(\cdot)^{-1}=\tau(Z)
\]
in $\mathcal{M}(T)\rtimes_qW$, where we view 
both $\mathbb{H}(k,q)$ and $\mathbb{H}(k^d,q)$ as subalgebras of 
$\mathcal{M}(T)\rtimes_qW$.
\end{prop}
\subsection{The nonsymmetric basic hypergeometric function}
The nonsymmetric Mac\-do\-nald-Mehta weight 
$\Xi(\cdot)=\Xi(\cdot;k,q): \Lambda\rightarrow\mathbb{C}$ is
\[
\Xi(\lambda;k,q):=\frac{G_{k^{\tau d},q}(\gamma_{\lambda,\tau})}
{G_{k^{\tau d},q}(\gamma_{0,\tau})N(\lambda;k^{d\tau},q)},
\]
where $\gamma_{\lambda,\tau}:=\gamma_\lambda(k^\tau,q)$,
$k^{\tau d}=(k^\tau)^d$ and $k^{d\tau}=(k^d)^\tau$ (see \cite[\S 7]{CM}
in the reduced case and \cite[\S 6.1]{St} in the nonreduced case). 
We have $\Xi(0;k,q)=1$ and
\begin{equation}\label{muproperties}
\begin{split}
\Xi(\lambda;k,q)&=\Xi(\lambda;k^d,q),\\
\Xi(-w_0\lambda;k,q)&=\Xi(\lambda;k,q)
\end{split}
\end{equation}
for all $\lambda\in\Lambda$.
Due to the factor $G_{k^{\tau d},q}(\gamma_{\lambda,\tau})$ 
in the weight $\Xi(\lambda)$,
the discrete Macdonald-Metha integral $M:=M(k,q)$ defined by
\[
M:=G(\gamma_0;k^{d\tau},q)G(\gamma_{0,d};k^\tau,q)\sum_{\lambda\in\Lambda}
\Xi(\lambda;k,q)
\]
is convergent. It can be evaluated explicitly, see \cite[Thm 1.1]{CM}
in the reduced case and \cite[Prop. 6.1]{St2} 
in the nonreduced case. It will play
the role of normalization constant for the 
(nonsymmetric) basic hypergeometric
function.

For $i\in\{1,\ldots,n\}$ we denote the simple
root $-w_0\alpha_i$ by $\alpha_{i^*}$. 
Write $\mathbb{K}$ for the field of meromorphic
functions on $T\times T$.
\begin{thm}\label{Ekernel}
{\bf (i)} There exists a unique antiisomorphism $\xi: \mathbb{H}(k,q)\rightarrow
\mathbb{H}((k^d)^{-1},q^{-1})$ satisfying
\begin{equation*}
\begin{split}
\xi(T_i)&=T_{i^*}^{-1},\qquad\qquad\quad 1\leq i\leq n,\\
\xi(Y^\lambda)&=X^{-w_0\lambda},\qquad\qquad \lambda\in\Lambda,\\
\xi(X^\lambda)&=T_{w_0}Y^\lambda T_{w_0}^{-1},\qquad\,\,\, \lambda\in\Lambda.
\end{split}
\end{equation*}
{\bf (ii)} There exists a unique $\mathcal{E}(\cdot,\cdot)
=\mathcal{E}(\cdot,\cdot;k,q)
\in\mathbb{K}$ satisfying
\begin{enumerate}
\item  $(t,\gamma)\mapsto G_{k^\tau,q}(t)^{-1}G_{k^{d\tau},q}(\gamma)^{-1}
\mathcal{E}(t,\gamma;k,q)$ is a holomorphic function on $T\times T$,
\item $\pi_{k,q}^t(Z)\mathcal{E}=
\pi_{(k^d)^{-1},q^{-1}}^\gamma(\xi(Z))\mathcal{E}$
for all $Z\in\mathbb{H}(k,q)$, where $\pi^t(Z)$ and
$\pi^\gamma(\xi(Z))$ are the actions of $\pi(Z)$ and
$\pi(\xi(Z))$ on the first and second torus variable respectively, 
\item $\mathcal{E}(\gamma_{0,d},\gamma_0)=1$.
\end{enumerate}
Explicitly,
\[
\mathcal{E}(t,\gamma;k,q)=M_{k,q}^{-1}G_{k^\tau,q}(t)G_{k^{d\tau},q}(\gamma)
\sum_{\lambda\in\Lambda}\Xi(\lambda;k,q)E(\gamma_{-w_0\lambda,\tau};t;k^\tau,q)
E^\prime(\gamma_{\lambda,d\tau}^{-1};\gamma;k^{d\tau},q)
\]
with $\gamma_{\lambda,d\tau}:=\gamma_\lambda(k^{d\tau},q)$.
The sum converges normally for $(t,\gamma)$ in compacta of
$T\times T$.
\end{thm}
\begin{proof}
{\bf (i)} Set $(If)(t):=f(t^{-1})$ for $f\in\mathbb{C}(T)$.
There exists a unique algebra isomorphism $\eta: \mathbb{H}(k,q)
\overset{\sim}{\longrightarrow}\mathbb{H}(k^{-1},q^{-1})$ such that 
$\pi_{k,q}(Z)\circ I=I\circ \pi_{k^{-1},q^{-1}}(\eta(Z))$ for all 
$Z\in\mathbb{H}(k,q)$ (both sides of the identity
viewed as operators on $\mathbb{C}(T)$).
Then
$\xi(Z)=T_{w_0}(\delta(\eta(Z)))T_{w_0}^{-1}$
with $\delta: \mathbb{H}(k^{-1},q^{-1})\overset{\sim}{\longrightarrow}
\mathbb{H}((k^d)^{-1},q^{-1})$ the linear duality antiisomorphism 
mapping $T_i$ to $T_i^{-1}$ ($1\leq i\leq n$), $X^\lambda$ to
$Y^{-\lambda}$ and $Y^\lambda$ to $X^{-\lambda}$ ($\lambda\in\Lambda$). 
See \cite[\S 3.3.2]{C} in the reduced case and \cite{Sa} in the nonreduced case
for further details on 
the duality antiisomorphism $\delta$, as well as \cite{Ion,Ha}.\\
{\bf (ii)} A nonsymmetric
kernel function $\widetilde{\mathcal{E}}\in\mathcal{M}(T\times T)$
was defined and studied by Cherednik \cite[\S 5]{CM} in the reduced
case (denoted in \cite{CM} as $\mathcal{E}_{q^{-1}}$)
and by the author \cite[\S 5]{St2} in the nonreduced case
(denoted in \cite{St2} as $\mathfrak{E}_\ddagger$).
Its transformation property with respect
to the actions of the double affine Hecke algebra is
\[
\pi_{k^{-1},q^{-1}}^t(Z)\widetilde{\mathcal{E}}=
\pi_{(k^d)^{-1},q^{-1}}^\gamma(\delta(Z))\widetilde{\mathcal{E}},
\qquad \forall\, Z\in\mathbb{H}(k^{-1},q^{-1}).
\]
Our kernel $\mathcal{E}$ can be expressed
in terms of $\widetilde{\mathcal{E}}$
by
\[
\mathcal{E}(t,\gamma)=\bigl(\pi_{(k^d)^{-1},q^{-1}}^\gamma(T_{w_0})
\widetilde{\mathcal{E}}\bigr)(t^{-1},\gamma)
\]
up to normalization,
cf. the proof of {\bf (i)}.
\end{proof}
\begin{defi}
We call $\mathcal{E}(\cdot,\cdot)=\mathcal{E}(\cdot,\cdot;k,q)$
the nonsymmetric basic hypergeometric function associated
to the triple $(R_0,\Delta_0,\Lambda)$.
\end{defi}
\begin{rema}
As already noted in the proof of Theorem \ref{Ekernel}, 
the definition of 
the nonsymmetric basic hypergeometric function $\mathcal{E}$
differs slightly from the definitions of the kernel functions in
\cite{CM,St2,CWhit}.
With our definition of $\mathcal{E}$ the
connection with meromorphic solutions 
of the bispectral quantum Knizhnik-Zamolodchikov equations will
be more transparent (see Subsection \ref{BqKZsection}). The difference 
between the definitions disappears upon 
symmetrization, cf. Subsection \ref{bhf}.
\end{rema}
Note that $\mathcal{E}(\cdot,\gamma)$ for $\gamma\in T$
such that $G_{k^{d\tau},q}(\gamma)\not=0$ is a meromorphic solution of the
spectral problem
\[
\pi_{k,q}(p(Y))f=(w_0p)(\gamma^{-1})f\qquad \forall\, p\in\mathbb{C}[T].
\]

In view of Theorem \ref{Ekernel} we actually have 
$\mathcal{E}\in\mathcal{V}=\mathcal{V}_{k,q}$,
where
\begin{equation}\label{mathcalS}
\mathcal{V}_{k,q}:=\{f\in\mathbb{K} \,\, | \,\,
\pi_{k,q}^t(Z)f=\pi_{(k^d)^{-1},q^{-1}}^\gamma(\xi(Z))f\quad \forall\,
Z\in\mathbb{H} \}.
\end{equation}
$\mathcal{V}$ is a vectorspace over $\mathbb{F}^{W_0\times W_0}$, with
$\mathbb{F}\subset\mathbb{K}$ the subfield
\begin{equation}\label{F}
\mathbb{F}:=\{f\in\mathbb{K}\,\, | \,\, f(q^\lambda t,q^{\lambda^\prime}\gamma)=
f(t,\gamma)\quad\forall\,\lambda,\lambda^\prime\in\Lambda \}
\end{equation}
of quasiconstant meromorphic functions on $T\times T$. 
\begin{prop}\label{propduality}
{\bf (i)} The involution $\iota$ of $\mathbb{K}$, defined by
$(\iota f)(t,\gamma)=f(\gamma^{-1},t^{-1})$, restricts to a complex
linear isomorphism
\[
\iota|_{\mathcal{V}_{k,q}}: \mathcal{V}_{k,q}\overset{\sim}{\longrightarrow}
\mathcal{V}_{k^d,q}.
\]
{\bf (ii)} The nonsymmetric basic hypergeometric function $\mathcal{E}$ is
selfdual,
\[
\iota\bigl(\mathcal{E}(\cdot,\cdot;k,q)\bigr)=
\mathcal{E}(\cdot,\cdot;k^d,q).
\]
\end{prop}
\begin{proof}
{\bf (i)} Let $f\in\mathcal{V}_{k,q}$ and set $g:=\iota f$.
Recall the isomorphism $\eta$ from the proof of Theorem \ref{Ekernel}.
Denote $\eta_d$ for the isomorphism $\eta$ with respect to dual
parameters $(k^d,q)$.
Then
\[
\pi_{k^d,q}^t(Z)g=\pi_{k^{-1},q^{-1}}(\widetilde{\xi}(Z))g,\qquad
\forall\,Z\in\mathbb{H}(k^d,q)
\]
with antiisomorphism
$\widetilde{\xi}=\eta^{-1}\circ\xi^{-1}\circ\eta_d:
\mathbb{H}(k^d,q)\rightarrow\mathbb{H}(k^{-1},q^{-1})$. 
Since
$\eta(T_i)=T_i^{-1}$, $\eta(X^\lambda)=X^{-\lambda}$
and $\eta(Y^\lambda)=T_{w_0}Y^{w_0\lambda}T_{w_0}^{-1}$ for $1\leq i\leq n$
and $\lambda
\in\Lambda$ (cf. \cite[Prop. 3.2.2]{C}), 
$\widetilde{\xi}$ is the antiisomorphism
$\xi$ with respect to dual parameters $(k^d,q)$. Hence
$g\in\mathcal{V}_{k^d,q}$.\\
{\bf (ii)} It follows from the explicit
series expansion of $\mathcal{E}$, \eqref{conjugationsymm} and
\eqref{muproperties} that
\[
\iota(\mathcal{E}(\cdot,\cdot;k,q))=
\pi_{k^d,q}^t(T_{w_0})\pi_{k^{-1},q^{-1}}^\gamma(T_{w_0})
\mathcal{E}(\cdot,\cdot;k^d,q).
\]
But this equals $\mathcal{E}(\cdot,\cdot;k^d,q)$ since
$\xi(T_{w_0})=T_{w_0}^{-1}$.
\end{proof}
\subsection{The basic hypergeometric function}\label{bhf}

We first recall some well known facts about symmetric Macdonald-Koornwinder
polynomials from e.g. \cite{C,M,N,Sa,St2}.
For $p\in\mathbb{C}[T]^{W_0}$ we decompose the $q$-difference reflection
operator $\pi_{k,q}(p(Y))$
associated to the central element $p(Y)\in Z(H(k^\bullet))$ as
\[
\pi_{k,q}(p(Y))=
\sum_{w\in W_0}D_{p,w}^{k,q}w_q,\qquad
D_{p,w}^{k,q}\in\mathbb{C}(T)\rtimes_q\tau(\Lambda).
\]
The Macdonald operator $D_p=D_p^{k,q}$
associated to $p$ is defined by
\[
D_p^{k,q}:=\sum_{w\in W_0}D_{p,w}^{k,q}.
\]
The Macdonald operators $D_p$ ($p\in\mathbb{C}[T]^{W_0}$)
are pairwise commuting, $W_0$-equivariant, scalar $q$-difference operators
(see, e.g., \cite[Lem. 2.7]{LS}). 
Explicit expressions of $D_p$ can be given
for special choices of $p\in\mathbb{C}[T]^{W_0}$,
in which case they reduce 
to the original definitions of the Macdonald, Koornwinder
and Ruijsenaars $q$-difference operators from \cite{Mpol},
\cite{Ko} and \cite{R} respectively (see \cite[\S 4.4]{M}).

The idempotent
\[
C_+:=
\frac{1}{\sum_{w\in W_0}k_w^2}\sum_{w\in W_0}k_wT_w\in H(k^\bullet)
\]
satisfies $T_iC_+=k_iC_+=C_+T_i$ for $1\leq i\leq n$
(we do not specify the $k^\bullet$-dependence of $C_+$, it
will always be clear from the context). It follows that
$\pi_{k,q}(C_+): \mathcal{M}(T)\rightarrow\mathcal{M}(T)$ is a
projection operator with image $\mathcal{M}(T)^{W_0}$.
Consequently, if $f\in\mathcal{M}(T)$ satisfies
the $q$-difference reflection equations
\[
\pi_{k,q}(p(Y))f=p(\gamma^{-1})f\qquad \forall\, p\in\mathbb{C}[T]^{W_0}
\]
for some $\gamma\in T$, then $f_+:=\pi_{k,q}(C_+)f\in\mathcal{M}(T)^{W_0}$
satisfies
\[
D_pf_+=p(\gamma^{-1})f_+\qquad \forall\, p\in\mathbb{C}[T]^{W_0}.
\]
In particular, the {\it normalized symmetric Macdonald-Koornwinder polynomial}
\[
E_+(\gamma_\lambda;\cdot):=\pi_{k,q}(C_+)E(\gamma_\lambda;\cdot)\in
\mathbb{C}[T]^{W_0},\qquad \lambda\in\Lambda^-
\]
satisfies
\[
D_p\bigl(E_+(\gamma_\lambda;\cdot)\bigr)=
p(\gamma_\lambda^{-1})E_+(\gamma_\lambda;\cdot)\qquad \forall\, 
p\in\mathbb{C}[T]^{W_0}.
\]
The monic symmetric Macdonald-Koornwinder polynomial
$P_\lambda^+(\cdot)=P_\lambda^+(\cdot;k,q)\in\mathbb{C}[T]^{W_0}$
($\lambda\in\Lambda^-$) is the 
renormalization of $E_+(\gamma_\lambda;\cdot)$ having an expression 
\[
P_\lambda^+(t)=\sum_{\mu\in Q_+}d_\mu t^{w_0\lambda-\mu}
\]
in monomials with leading coefficient $d_0=1$. Then
$E_+(\gamma_\lambda;\cdot)=P_\lambda^+(\gamma_{0,d})^{-1}P_\lambda^+(\cdot)$,
since $E_+(\gamma_\lambda;\gamma_{0,d})=1$. 
Selfduality and the evaluation formula for the symmetric 
Macdonald-Koorn\-win\-der follow from the corresponding results
for the nonsymmetric Macdonald-Koorn\-win\-der polynomials by standard 
symmetrization arguments. Alternatively they can be derived
from the asymptotic analysis of the bispectral quantum Knizhnik-Zamolodchikov
equations, see Remark \ref{consequenceremark}.

Before symmetrizing the nonsymmetric basic hypergeometric function
$\mathcal{E}$, we first introduce and analyze
the natural space it will be contained
in, cf. \cite[Def. 6.13]{vMSt} for $\textup{GL}_m$ 
and \cite[Def. 6.4]{vM} for the reduced case.
\begin{defi}
We set $\mathcal{U}:=\mathcal{U}_{k,q}$
for the $\mathbb{F}$-vector space of meromorphic functions
$f$ on $T\times T$ satisfying
\begin{equation}\label{bisp}
\begin{split}
\bigl(D_p^tf\bigr)(t,\gamma)&=
p(\gamma^{-1})f(t,\gamma)\\
\bigl(\widetilde{D}_p^\gamma f\bigr)(t,\gamma)&=
p(t)f(t,\gamma)
\end{split}
\end{equation}
for all $p\in\mathbb{C}[T]^{W_0}$,
where $\widetilde{D}_p=D_p^{(k^d)^{-1},q^{-1}}$. The superindices
$t$ and $\gamma$ indicate that the $q$-difference operator is acting
on the first and second torus component respectively.
\end{defi}
Note that $\mathcal{U}$ is a $W_0\times W_0$-invariant subspace of
$\mathbb{K}$. 

View 
$\pi_{k,q}^t(C_+)$ and $\pi_{(k^d)^{-1},q^{-1}}^\gamma(C_+)$ as
projection operators on $\mathbb{K}$. Their images are 
$\mathbb{K}^{W_0\times\{1\}}$ and $\mathbb{K}^{\{1\}\times W_0}$
respectively.
\begin{lem}\label{po}
{\bf (i)} 
The restrictions of the projection operators $\pi_{k,q}^t(C_+)$ and
$\pi_{(k^d)^{-1},q^{-1}}^\gamma(C_+)$ to $\mathcal{V}_{k,q}$ coincide, and map into
$\mathcal{U}_{k,q}^{W_0\times W_0}$.\\
{\bf (ii)}  The involution $\iota$ of $\mathbb{K}$ restricts to a
complex linear isomorphism
\[
\iota|_{\mathcal{U}_{k,q}}: \mathcal{U}_{k,q}\overset{\sim}{\longrightarrow}
\mathcal{U}_{k^d,q}.
\]
{\bf (iii)} For all $f\in\mathcal{V}$,
\[
\pi_{k^d,q}^t(C_+)(\iota f)=\iota(\pi_{k,q}^t(C_+)f).
\]
\end{lem}
\begin{proof}
Since $\xi(C_+)=C_+$, the restrictions of $\pi_{k,q}^t(C_+)$
and $\pi_{(k^d)^{-1},q^{-1}}^\gamma(C_+)$ to $\mathcal{V}_{k,q}$ coincide.
Let $f\in\mathcal{V}_{k,q}$ and set $f_+:=\pi_{k,q}^t(C_+)f$.
Since $\xi(p(Y))=p(X^{-1})$ for all $p\in\mathbb{C}[T]^{W_0}$
and since the projection operator $\pi_{k,q}(C_+)$ on $\mathbb{K}$
has range $\mathbb{K}^{W_0\times\{1\}}$, the meromorphic function
$f_+$ satisfies the first set of equations
from \eqref{bisp}.
For the second set of equations of \eqref{bisp} note that 
$f_+$ is $W_0$-invariant in the second torus component
since $f_+=\pi_{(k^d)^{-1},q^{-1}}^\gamma(C_+)f$.
Then for $p\in\mathbb{C}[T]^{W_0}$,
\begin{equation*}
\begin{split}
D_p^{(k^d)^{-1},q^{-1}}f_+&=\pi_{(k^d)^{-1},q^{-1}}^\gamma(C_+p(Y))f\\
&=\pi_{(k^d)^{-1},q^{-1}}^\gamma(C_+)\pi_{(k^d)^{-1},q^{-1}}(T_{w_0}p(Y)T_{w_0}^{-1})f\\
&=\pi_{k,q}^t(p(X))\pi_{(k^d)^{-1},q^{-1}}^\gamma(C_+)f\\
&=\pi_{k,q}^t(p(X))f_+
\end{split}
\end{equation*}
since $\xi(p(X))=T_{w_0}p(Y)T_{w_0}^{-1}$. Hence 
$f_+\in\mathcal{U}_{k,q}^{W_0\times W_0}$, proving {\bf (i)}.
Part {\bf (iii)}
follows from {\bf (i)} and the fact that
\[
\iota\circ\pi_{k,q}^t(C_+)=\pi_{k^{-1},q^{-1}}^\gamma(C_+)
\]
(which in turn follows from the fact that $\eta(C_+)=C_+$).
It remains to prove
{\bf (ii)}. It suffices to show that $I\circ D_p^{k,q}\circ I=
D_p^{k^{-1},q^{-1}}$ for $p\in\mathbb{C}[T]^{W_0}$ as endomorphism of
$\mathcal{M}(T)$, where $(Ig)(t):=g(t^{-1})$. This follows from
$\eta(p(Y))=p(Y)$ for $p\in\mathbb{C}[T]^{W_0}$,
cf. the proof of Proposition \ref{propduality} (see also \cite[Lem. 6.2]{vM}).
\end{proof}
The nonsymmetric basic hypergeometric function $\mathcal{E}$ associated to
$(R_0,\Delta_0,\Lambda)$, being a distinguished element of $\mathcal{V}$,
thus gives rise to a distinguished 
$W_0\times W_0$-symmetric meromorphic solution of the bispectral problem
\eqref{bisp}:
\begin{defi}
We call $\mathcal{E}_+(\cdot,\cdot)=
\mathcal{E}_+(\cdot,\cdot;k,q):=\pi_{k,q}^t(C_+)\mathcal{E}(\cdot,\cdot;k,q)
\in\mathcal{U}^{W_0\times W_0}$
the basic hypergeometric function 
associated to the triple $(R_0,\Delta_0,\Lambda)$.
\end{defi}
In the reduced case $\mathcal{E}_+$ is Cherednik's \cite{CM, CWhit} 
global spherical function. 
In the nonreduced case $\mathcal{E}_+$
was defined by the author in \cite{St2}.

We list the key properties of the basic hypergeometric function 
in the following theorem. 
\begin{thm}\label{QSP}
{\bf (i)} Explicit series expansion,
\[
\mathcal{E}_+(t,\gamma;k,q)=M_{k,q}^{-1}G_{k^\tau,q}(t)G_{k^{d\tau},q}(\gamma)
\sum_{\lambda\in\Lambda^-}\Xi^+(\lambda;k,q)E_+(\gamma_{\lambda,\tau};t;k^\tau,q)
E_+(\gamma_{\lambda,d\tau};\gamma^{-1};k^{d\tau},q)
\]
with $\Xi^+(\lambda;k,q):=\sum_{\nu\in W_0\lambda}\Xi(\nu;k,q)$.
The sum converges normally for $(t,\gamma)$ in compacta
of $T\times T$.\\
{\bf (ii)} Inversion symmetry,
\[
\mathcal{E}_+(t,\gamma;k,q)=\mathcal{E}_+(t^{-1},\gamma^{-1};k,q).
\]
{\bf (iii)} Duality,
\[
\iota\bigl(\mathcal{E}_+(\cdot,\cdot;k,q)\bigr)=
\mathcal{E}_+(\cdot,\cdot;k^d,q).
\]
{\bf (iv)} Reduction to symmetric Macdonald-Koornwinder polynomials, 
\[
\mathcal{E}_+(t,\gamma_\lambda;k,q)=E_+(\gamma_{\lambda_-};t;k,q)\qquad 
\forall\,\lambda\in\Lambda,
\]
with $\lambda_-\in\Lambda^-$ the unique antidominant weight in
the orbit $W_0\lambda$.
\end{thm}
\begin{proof} We only sketch the proof.
For detailed proofs see \cite{CM}
in the reduced case and \cite{St2} in the nonreduced case.\\
{\bf (i)} This follows from rather standard symmetrization arguments, using
the fact that $\pi(C_+)E(\gamma_\lambda;\cdot)$
only depends on the orbit $W_0\lambda$ of $\lambda$ and that
\begin{equation}\label{symminvert}
E^\prime_+(\gamma_\lambda^{-1};t)=E_+(\gamma_{-w_0\lambda};t^{-1})
\end{equation}
for $\lambda\in\Lambda^-$, where
$E^\prime_+(\gamma_\lambda^{-1};\cdot):=
\pi_{k^{-1},q^{-1}}(C_+)E^\prime(\gamma_\lambda^{-1};\cdot)$. Formula
\eqref{symminvert} is a direct consequence of \eqref{conjugationsymm}.\\
{\bf (ii)} This follows from {\bf (i)} and the formula
$E_+(\gamma_\lambda;t^{-1})=E_+(\gamma_{-w_0\lambda};t)$ for 
$\lambda\in\Lambda^-$. The latter formula is a consequence of \eqref{symminvert}
and the fact that $E_+^\prime(\gamma_\lambda^{-1};t)=E_+(\gamma_\lambda;t)$
for $\lambda\in\Lambda^-$ (see, e.g., \cite[(5.3.2)]{M}).\\ 
{\bf (iii)} This follows from {\bf (i)} and the self-duality
$\Xi^+(\lambda;k,q)=\Xi^+(\lambda;k^d,q)$ of the
weight $\Xi^+$. Alternatively, use Proposition \ref{propduality} and 
Lemma \ref{po}{\bf (iii)}.\\
{\bf (iv)} This is Cherednik's generalization of the Shintani-Casselman-Shalika
formula in the reduced case (see \cite[(7.13)]{CM},\cite[(3.11)]{CWhit}).
For the nonreduced case, see \cite[Thm. 6.15(d)]{St2}.
\end{proof}

\section{Basic Harish-Chandra series}\label{HCsection}

In this section we generalize and analyze the basic Harish-Chandra
series from \cite{vMSt} ($\textup{GL}_m$ case) and from \cite{vM} 
(reduced case). The basic Harish-Chandra series is a 
$q$-analog of the Harish-Chandra series solution of
the Heckman-Opdam hypergeometric system associated to root systems 
(see \cite[Part I, Chpt. 4]{HS} and references therein). 

Our approach differs from the classical treatment, 
in the sense that we construct, following \cite{vMSt,vM}, the basic
Harish-Chandra series as matrix coefficient of 
a power series solution of a bispectral extension of 
Cherednik's \cite{CKZ,CKZ2} quantum affine Knizhnik-Zamolodchikov 
(KZ) equations associated to the minimal principal series
of $H(k^\bullet)$ (the extension being given by a compatible set of equations
acting on the central character of the minimal principal series representation).
This is essential for two reasons:
\begin{enumerate}
\item Convergence issues:
formal power series solutions of the (bispectral) quantum 
KZ equation are easily seen to converge deep in the asymptotic
sector, in contrast to formal power series
solutions of the spectral problem for the Macdonald $q$-difference operators.
\item The formal power series solution of the bispectral quantum KZ equation
gives rise to a selfdual, globally meromorphic $q$-analog of the classical
Harish-Chandra series. The selfduality plays an important role in our proof
of the $c$-function expansion of the basic hypergeometric function in Section
\ref{c}.
\end{enumerate}
Our approach also gives new proofs of the 
selfduality and the evaluation formula for the symmetric Macdonald-Koornwinder
polynomials (see Remark \ref{consequenceremark}).

\subsection{Bispectral quantum Knizhnik-Zamolodchikov equations}
\label{BqKZsection}
In this subsection we show that the space $\mathcal{V}$ (see \eqref{mathcalS})
is isomorphic to
the space of solutions of a bispectral extension of the quantum affine
Knizhnik-Zamolodchikov (KZ) equations. 

We will first introduce the bispectral quantum KZ equations, following
and extending \cite{vMSt,vM}. 
Tensor products and endomorphism spaces will be over $\mathbb{C}$ unless
stated explicitly otherwise. Let $\chi: R_0\rightarrow\{0,1\}$ 
be the characteristic function of $R_0^-$.
Set $M:=\bigoplus_{w\in W_0}\mathbb{C}v_w$. Define elements
\[
C_{(w,1)}^{k,q}, C_{(1,w)}^{k,q}\in\mathbb{K}\otimes
\textup{End}(M)
\]
for the generators $w=s_i$, $w=u(\lambda)$ ($0\leq i\leq n$ and 
$\lambda\in\Lambda_{min}^+$)
of the extended affine Weyl group $W$ by 
\begin{equation*}
\begin{split}
C_{(s_0,1)}^{k,q}(t,\gamma)v_w&:=
\frac{\gamma^{-w^{-1}\theta}v_{s_\theta w}}{k_0c_0(t;k^{-1},q)}+
\left(\frac{c_0(t;k^{-1},q)-k_0^{-2\chi(w^{-1}\theta)}}
{c_0(t;k^{-1},q)}\right)v_w,\\
C_{(s_i,1)}^{k,q}(t,\gamma)v_w&:=\frac{v_{s_iw}}{k_ic_i(t;k^{-1},q)}+
\left(
\frac{c_i(t;k^{-1},q)-k_i^{-2\chi(-w^{-1}\alpha_i)}}{c_i(t;k^{-1},q)}\right)v_w,\\
C_{(u(\lambda),1)}^{k,q}(t,\gamma)v_w&:=\gamma^{w^{-1}w_0\lambda}v_{v(\lambda)^{-1}w}
\end{split}
\end{equation*}
for $1\leq i\leq n$, $\lambda\in\Lambda_{min}^+$ and $w\in W_0$, 
where $v(\lambda)\in W_0$ is the element of minimal length
such that $v(\lambda)\lambda\in\Lambda^-$, and 
\begin{equation*}
\begin{split}
C_{(1,s_0)}^{k,q}(t,\gamma)v_w&:=
\frac{t^{w\theta}v_{ws_\theta}}{k_0^dc_0(\gamma^{-1};(k^d)^{-1},q)}+
\left(\frac{c_0(\gamma^{-1};(k^d)^{-1},q)-(k_0^d)^{-2\chi(w\theta)}}
{c_0(\gamma^{-1};(k^d)^{-1},q)}\right)v_w,\\
C_{(1,s_i)}^{k,q}(t,\gamma)v_w&:=\frac{v_{ws_i}}
{k_i^dc_i(\gamma^{-1};(k^d)^{-1},q)}+
\left(
\frac{c_i(\gamma^{-1};(k^d)^{-1},q)-(k_i^d)^{-2\chi(-w\alpha_i)}}
{c_i(\gamma^{-1};(k^d)^{-1},q)}\right)v_w,\\
C_{(1,u(\lambda))}^{k,q}(t,\gamma)v_w&:=t^{-ww_0\lambda}v_{wv(\lambda)}.
\end{split}
\end{equation*}
The following theorem is \cite[Cor. 3.4 \& Lem. 4.3]{vMSt} 
in the $\textup{GL}_m$-case
and \cite[Cor. 3.8 \& Lem. 4.3]{vM} 
in the reduced case. The extension to the nonreduced case
is straightforward.
\begin{thm}
There exists a unique left $W\times W$-action 
$((w_1,w_2),g)\mapsto\nabla^{k,q}((w_1,w_2))g$
on $\mathbb{K}\otimes M$
satisfying
\begin{equation*}
\begin{split}
\nabla^{k,q}(w,1)g&=C_{(w,1)}^{k,q}w_q^tg\\
\nabla^{k,q}(1,w)g&=C_{(1,w)}^{k,q}w_{q^{-1}}^\gamma g
\end{split}
\end{equation*}
for $g\in\mathbb{K}$,
$w=s_j$ ($0\leq j\leq n$) and $w=u(\lambda)$ ($\lambda\in\Lambda_{min}^+$),
where
\[(w_q^tg)(t,\gamma)=g(w_q^{-1}t,\gamma),\qquad
(w_{q^{-1}}^\gamma g)(t,\gamma)=g(t,w_{q^{-1}}^{-1}\gamma).
\]
\end{thm}
We say that $g\in\mathbb{K}\otimes M$ satisfies the bispectral
quantum Knizhnik-Zamolodchikov equations if $g$ is a solution of the
compatible system
\begin{equation}\label{BqKZ}
\nabla(\tau(\lambda),\tau(\lambda^\prime))g=g\qquad \forall\, 
(\lambda,\lambda^\prime)\in\Lambda\times\Lambda
\end{equation}
of $q$-difference equations. Restricting the equations
\eqref{BqKZ} to $\Lambda\times\{0\}$
and fixing the second torus variable $\gamma\in T$ gives, in the reduced case,
Cherednik's \cite{CKZ,CKZ2} quantum affine KZ equation associated to the minimal
principal series representation of $H(k^\bullet)$ with central character 
$\gamma$.
\begin{defi}
We write $\mathcal{K}=\mathcal{K}_{k,q}$ 
for the $\mathbb{F}$-vector space consisting of 
$g\in\mathbb{K}\otimes M$ satisfying the bispectral
quantum KZ equations \eqref{BqKZ}. 
\end{defi}
Note that $\mathcal{K}$ is a
$W_0\times W_0$-module, with action the restriction of $\nabla$
to $W_0\times W_0$. 

Let $\sigma$ be the complex linear automorphism of 
$\mathbb{K}\otimes M$
defined by 
\[
\sigma(f\otimes v_w):=\iota f\otimes v_{w^{-1}}.
\]
Then
\begin{equation}\label{duality}
\sigma\circ\nabla^{k,q}((w,w^\prime))=\nabla^{k^d,q}((w^\prime,w))\circ\sigma\qquad
\forall\, w,w^\prime\in W.
\end{equation}
In particular, $\sigma$ restricts to a complex linear isomorphism
$\mathcal{K}_{k,q}\overset{\sim}{\longrightarrow}\mathcal{K}_{k^d,q}$.

Define complex linear maps
\begin{equation*}
\begin{split}
\psi:=\psi_{k,q}:\, &\mathbb{K}\rightarrow \mathbb{K}\otimes M,\\
\phi:=\phi_{k,q}:\, &\mathbb{K}\otimes M\rightarrow\mathbb{K}
\end{split}
\end{equation*}
by 
\begin{equation*}
\begin{split}
&\psi(f):=\sum_{w\in W_0}(\pi^t_{k,q}(T_{ww_0})f)\otimes v_w,\\
&\phi\bigl(\sum_{w\in W_0}f_w\otimes v_w\bigr):=\frac{k_{w_0}^{-1}}
{\sum_{w\in W_0}k_w^{-2}}\sum_{w\in W_0}k_w^{-1}f_w.
\end{split}
\end{equation*} 
Note that $\phi\circ\psi=\pi^t_{k,q}(C_+)$. Recall from Lemma \ref{po}
that $\pi^t_{k,q}(C_+)$ restricts to a complex linear map
$\pi^t_{k,q}(C_+): \mathcal{V}\rightarrow\mathcal{U}^{W_0\times W_0}$.
It factorizes through the solutions space $\mathcal{K}^{W_0\times W_0}$
of the bispectral quantum KZ equations:
\begin{thm}\label{relspaces}
{\bf (i)} $\psi$ restricts to a $\mathbb{F}^{W_0\times W_0}$-linear
isomorphism $\psi: \mathcal{V}\overset{\sim}{\longrightarrow}
\mathcal{K}^{W_0\times W_0}$.\\
{\bf (ii)} $\phi$ restricts to an injective $W_0\times W_0$-equivariant
$\mathbb{F}$-linear map
$\phi: \mathcal{K}\hookrightarrow \mathcal{U}$.\\
{\bf (iii)} $\psi_{k^d,q}\circ\iota|_{\mathcal{V}_{k,q}}=
\sigma\circ\psi_{k,q}|_{\mathcal{V}_{k,q}}$ and 
$\phi_{k^d,q}\circ\sigma|_{\mathcal{K}_{k,q}^{W_0\times W_0}}=\iota\circ
\phi_{k,q}|_{\mathcal{K}_{k,q}^{W_0\times W_0}}$.
\end{thm}
\begin{proof}
{\bf (i)} The analogous statement in the reduced case for the 
usual quantum affine
KZ equations was proved in \cite[Thm. 4.9]{StKZ}. 
Its extension to the nonreduced
case is straightforward. The bispectral extension 
follows by a repetition of the arguments for the dual part of the quantum
KZ equations (i.e. the part acting on the second torus component).\\
{\bf (ii)} This is the bispectral extension of the difference
Cherednik-Matsuo correspondence
\cite[Thm. 3.4(a)]{CKZ}. See \cite[Thm. 6.16 \& Cor. 6.21]{vMSt} for the 
$\textup{GL}_m$-case and
\cite[Thm. 6.6]{vM} for the reduced case (the injectivity follows from
the asymptotic analysis of the bispectral quantum KZ equations, which
we will also recall in Subsection \ref{Abqkz}). The extension to the nonreduced
case is straightforward. An alternative approach is to extend the techniques
from \cite[\S 5]{StKZ} to the present bispectral (and nonreduced) setting.\\
{\bf (iii)} Using $\eta(T_{ww_0})=T_{w_0w^{-1}}^{-1}$
for $w\in W_0$ it follows that
\[
\psi_{k^d,q}(\iota f)=\sigma\bigl(\sum_{w\in W_0}
\pi_{(k^d)^{-1},q^{-1}}^\gamma(T_{w_0w}^{-1})f\otimes v_w\bigr)
\]
for $f\in\mathbb{K}$. The first part then follows from the observation that
\[
\psi_{k,q}(f)=\sum_{w\in W_0}\pi_{(k^d)^{-1},q^{-1}}^\gamma(T_{w_0w}^{-1})f\otimes v_w
\]
if $f\in\mathcal{V}_{k,q}$, 
since $\xi(T_{ww_0})=T_{w_0w}^{-1}$ for $w\in W_0$. For the second equality
let $f\in\mathcal{K}^{W_0\times W_0}$ and set $g=\psi^{-1}(f)\in\mathcal{V}$.
Then
\begin{equation*}
\begin{split}
\phi(\sigma(f))&=\pi^t(C_+)\psi^{-1}(\sigma(f))\\
&=\pi^t(C_+)(\iota g)\\
&=\iota(\pi^t(C_+)g)\\
&=\iota(\phi(f)),
\end{split}
\end{equation*}
where we use the first part of {\bf (ii)} for the second equality
and Lemma \ref{po}{\bf (iii)} for the third equality.
\end{proof}

\begin{cor}\label{imagephi}
$\mathcal{E}_+\in\phi(\mathcal{K}^{W_0\times W_0})$.
\end{cor}
\begin{proof}
$\mathcal{E}_+=\pi^t(C_+)\mathcal{E}=\phi(\psi(\mathcal{E}))$
and $\psi(\mathcal{E})\in\mathcal{K}^{W_0\times W_0}$ since
$\mathcal{E}\in\mathcal{V}$.
\end{proof}

\subsection{Asymptotically free solutions of the bispectral quantum KZ
equations}\label{Abqkz}

We recall the results on asymptotically free solutions of the bispectral
quantum KZ equations from \cite{vMSt} ($\textup{GL}_m$ case) and
\cite{vM} (reduced case). The extension to the nonreduced case presented
here follows from straightforward adjustments of the arguments of
\cite{vMSt,vM}.

Define $\mathcal{W}(\cdot,\cdot)=\mathcal{W}(\cdot,\cdot;k,q)\in\mathbb{K}$ by 
\[
\mathcal{W}(t,\gamma):=\frac{\vartheta(t(w_0\gamma)^{-1})}
{\vartheta(\gamma_0t)\vartheta(\gamma_{0,d}^{-1}\gamma)}.
\]
There is some flexibility in the choice of $\mathcal{W}(\cdot,\cdot)$.
The key properties we need it 
to satisfy, are the functional equations
\begin{equation}\label{frW}
\mathcal{W}(q^\lambda t,\gamma)=
\gamma_0^\lambda\gamma^{w_0\lambda}\mathcal{W}(t,\gamma),\qquad
\lambda\in\Lambda
\end{equation}
and the selfduality property
\[
\iota\bigl(\mathcal{W}(\cdot,\cdot;k,q)\bigr)=
\mathcal{W}(\cdot,\cdot;k^d,q).
\]
For $\epsilon>0$ set 
\[
B_\epsilon:=\{t\in T \,\, | \,\,
|t^{\alpha_i}|<\epsilon \quad\forall\,i\in\{1,\ldots,n\}\}
\]
and $B_\epsilon^{-1}:=\{t^{-1}\,\, | \,\, t\in B_\epsilon\}$.
\begin{thm}\label{AF}
There exists a unique $F(\cdot,\cdot)=
F(\cdot,\cdot;k,q)\in\mathcal{K}_{k,q}$ such that $F(t,\gamma)=
\mathcal{W}(t,\gamma)H(t,\gamma)$ with
$H(\cdot,\cdot)=H(\cdot,\cdot;k,q)\in\mathbb{K}\otimes M$ 
satisfying for $\epsilon>0$ sufficiently
small,
\[
H(t,\gamma)=\sum_{\mu,\nu\in Q_+}H_{\mu,\nu}t^{-\mu}\gamma^\nu\quad
(H_{\mu,\nu}\in M),\qquad H_{0,0}=v_{w_0}
\]
for $(t,\gamma)\in B_\epsilon^{-1}\times B_\epsilon$, with the series
converging normally for $(t,\gamma)$ in compacta of
$B_\epsilon^{-1}\times B_\epsilon$. 
\end{thm}
\begin{proof}
See \cite[Thm. 5.3]{vMSt} ($\textup{GL}_m$ case) 
and \cite[Thm. 5.4]{vM} (reduced case).
The proofs are based on the asymptotic analysis of compatible systems
of $q$-difference
equations using classical methods which go back to Birkhoff \cite{Bi}
(see the appendix of \cite{vMSt}). 
These results extend immediately to the present
setup if one restricts the bispectral quantum KZ equations
\eqref{BqKZ} to $\lambda,\lambda^\prime$
in the sublattice $\bigoplus_{i=1}^n\mathbb{Z}\varpi_i$
of $\Lambda$. But the resulting function $F(\cdot,\cdot)$ then
automatically satisfies \eqref{BqKZ} for all 
$\lambda,\lambda^\prime\in\Lambda$ due to the compatibility of the
bispectral quantum KZ equations \eqref{BqKZ}
(cf. the proof of \cite[Thm. 3.4]{StqHC}).
\end{proof}

For $a\in R^\bullet$ let $n_a(\cdot)=n_a(\cdot;k,q)$ be the rational function 
\begin{equation*}
n_a(t)=
\begin{cases}
1-k_a^{-2}t_q^a,\qquad &\hbox{if } 2a\not\in R,\\
(1-k_a^{-1}k_{2a}^{-1}t_q^a)(1+k_a^{-1}k_{2a}t_q^a),\qquad &\hbox{if } 
2a\in R.
\end{cases}
\end{equation*}
Note that $c_a(t;k^{-1},q)=n_a(t;k,q)/n_a(t;\mathbf{1},q)$ 
for $a\in R^\bullet$, with $\mathbf{1}$ the multiplicity function identically
equal to one. 
Let $\mathcal{L}(\cdot)=\mathcal{L}_{q}(\cdot)$ and 
$\mathcal{S}(\cdot)=\mathcal{S}_{k,q}(\cdot)$
be the holomorphic functions on $T$ defined by
\begin{equation*}
\mathcal{L}_{q}(t):=\prod_{\stackrel{\alpha\in R_0^+,}{r\in\mathbb{Z}_{>0}}}
n_{\alpha+r\frac{|\alpha|^2}{2}c}(t;\mathbf{1},q),\qquad
\mathcal{S}_{k,q}(t):=
\prod_{\stackrel{\alpha\in R_0^+,}{r\in\mathbb{Z}_{>0}}}n_{\alpha+
r\frac{|\alpha|^2}{2}c}(t;k,q).
\end{equation*}
We give the key properties of $F(\cdot,\cdot)$ in the following theorem.
The proof follows from straightforward adjustments of the arguments
in \cite{vMSt,vM} (which corresponds to the $\textup{GL}_m$ case
and reduced case respectively).
\begin{thm}\label{properties}
{\bf (i)} $F\in\mathcal{K}$ is selfdual:
$\sigma(F(\cdot,\cdot;k,q))=F(\cdot,\cdot;k^d,q)$.\\
{\bf (ii)} $\{\nabla(1,w)F\}_{w\in W_0}$ is a 
$\mathbb{F}$-basis of $\mathcal{K}$.\\
{\bf (iii)} $T\times T\ni (t,\gamma)\mapsto 
\mathcal{S}_{k,q}(t^{-1})\mathcal{S}_{k^d,q}(\gamma)H(t,\gamma;k,q)$
is holomorphic.\\
{\bf (iv)} For $\epsilon>0$ sufficiently small there exist unique holomorphic
$M$-valued functions $\Upsilon_\mu(\cdot)$ on $T$ ($\mu\in Q_+$) such that
\[
\mathcal{S}_{k^d,q}(\gamma)H(t,\gamma;k,q)=
\sum_{\mu\in Q_+}\Upsilon_\mu(\gamma)t^{-\mu}
\]
for $(t,\gamma)\in B_\epsilon^{-1}\times T$, with the series converging
normally for $(t,\gamma)$ in compacta of $B_\epsilon^{-1}\times T$.\\
{\bf (v)} $\Upsilon_0(\gamma)=\mathcal{L}_q(\gamma)v_{w_0}$. 
\end{thm}
{}From the third part of the theorem we conclude
\begin{cor}
Let $Z_{k,q}\subseteq T$ be the zero locus of $\mathcal{S}_{k,q}(\cdot)$
and set $Z_{k,q}^{-1}:=\{t^{-1} \, | \, t\in Z_{k,q}\}$.
Then $H(\cdot,\cdot;k,q)$ is holomorphic on $T\setminus Z_{k,q}^{-1}\times
T\setminus Z_{k^d,q}$.
\end{cor}
In the reduced case,
\[
Z_{k,q}=\{t\in T \,\, | \,\, t^\alpha\in k_\alpha^2q_{\alpha}^{-\mathbb{Z}_{>0}}
\,\,\,\textup{for some}\,\,\, \alpha\in R_0^+\}.
\]
In the nonreduced case,
\begin{equation*}
\begin{split}
Z_{k,q}=&\left\{t\in T \,\, | \,\, t^\alpha\in\{aq_\theta^{-2\mathbb{Z}_{>0}},
bq_\theta^{-2\mathbb{Z}_{>0}},cq_\theta^{-2\mathbb{Z}_{>0}},
dq_\theta^{-2\mathbb{Z}_{>0}}\}\,\,\,\textup{for some}\,\,\, \alpha\in R_{0,s}^+
\right.\\
&\left.\qquad\qquad\qquad\qquad\qquad\qquad\qquad
\quad \textup{or}\,\, t^\beta\in 
k_\vartheta^2q_{\vartheta}^{-\mathbb{Z}_{>0}}\,\,\,
\textup{for some}\,\,\, \beta\in R_{0,l}^+\right\},
\end{split}
\end{equation*}
where
\begin{equation}\label{AWpar}
\{a,b,c,d\}:=\{k_\theta k_{2\theta}, -k_\theta k_{2\theta}^{-1},
q_\theta k_0k_{2a_0}, -q_\theta k_0k_{2a_0}^{-1}\}.
\end{equation}

\subsection{Basic Harish-Chandra series}\label{bHCs}
Following \cite[\S 6.3]{vMSt} and \cite[\S 7]{vM} 
we have the following fundamental definition.
\begin{defi}\label{bHCdef}
The selfdual 
basic Harish-Chandra series $\Phi(\cdot,\cdot)=\Phi(\cdot,\cdot;k,q)
\in\mathcal{U}_{k,q}$ is defined by 
\begin{equation}\label{decPhi}
\Phi:=\phi(F)=\mathcal{W}\phi(H).
\end{equation}
\end{defi}
The properties of $F$ from Theorem \ref{properties}
(singularities, selfduality, leading term) can immediately be
transferred to the selfdual basic Harish-Chandra series $\Phi$. 
In particular, by Theorem \ref{relspaces}{\bf (iii)}
the selfduality of $F$ gives the selfduality of $\Phi$, 
\[
\iota\bigl(\Phi(\cdot,\cdot;k,q)\bigr)=\Phi(\cdot,\cdot;k^d,q).
\]

In the derivation of
the $c$-function expansion of the basic hypergeometric function
we initially make use of
the selfdual basic Harish-Chandra series.
To make the connection to the classical theory more 
transparent we will reformulate these
results in terms of a renormalization of $\Phi(t,\gamma)$
which is closer to the standard normalization of the
classical Harish-Chandra series (see the introduction). 
It is a $\gamma$-dependent
renormalization of $\Phi(t,\gamma)$, which also depends on a base
point $\eta\in T$ (indicating the choice of normalization of the prefactor). 
This renormalization of $\Phi$ breaks the duality symmetry.

To define the renormalized version of the basic Harish-Chandra series,
consider first the renormalization
$\widehat{H}(\cdot,\cdot)=\widehat{H}(\cdot,\cdot;k,q)\in
\mathbb{K}\otimes M$ of $H(\cdot,\cdot)$ given by
\[
\widehat{H}(t,\gamma):=
\frac{\mathcal{S}_{k^d,q}(\gamma)\sum_{w\in W_0}k_w^2}{\mathcal{L}_q(\gamma)}
H(t,\gamma).
\]
Note that for $\epsilon>0$ sufficiently small,
\[
\widehat{H}(t,\gamma)=\sum_{\mu\in Q_+}\widehat{\Upsilon}_\mu(\gamma)t^{-\mu}
\]
for $(t,\gamma)\in B_\epsilon^{-1}\times
\{\gamma\in T \,\, | \,\, \mathcal{L}_q(\gamma)\not=0\}$,
and $\phi(\widehat{\Upsilon}_0)\equiv 1$.

The monic basic Harish-Chandra series
$\widehat{\Phi}_\eta(\cdot,\cdot)=\widehat{\Phi}_\eta(\cdot,\cdot;k,q)$
with generic reference point $\eta\in T$ is now defined by 
\[
\widehat{\Phi}_\eta:=\widehat{\mathcal{W}}_\eta\phi(\widehat{H})
\]
with prefactor $\widehat{\mathcal{W}}_\eta(\cdot,\cdot)=
\widehat{\mathcal{W}}_\eta(\cdot,\cdot;k,q)
\in\mathbb{K}$ defined as follows.
Let 
\[\rho_s^\vee:=\frac{1}{2}\sum_{\beta\in R_{0,s}^+}\beta^\vee
\]
with $R_{0,s}^+\subset R_0^+$ 
the subset of positive short roots. For $x\in\mathbb{R}_{>0}$
let $x^{\rho_s^\vee}\in T$ be the torus element
$\lambda\mapsto x^{(\rho_s^\vee,\lambda)}$ ($\lambda\in\Lambda$). 
Then $\widehat{\mathcal{W}}_\eta$ is defined to be
\[
\widehat{\mathcal{W}}_\eta(t,\gamma)=\frac{\widehat{\mathcal{W}}(t,\gamma)}
{\widehat{\mathcal{W}}(\eta\gamma_{0,d},\gamma)}
\]
with
\[
\widehat{\mathcal{W}}(t,\gamma)=\frac{\vartheta\bigl(\gamma_0^{-1}
(k_0^{-1}k_{2a_0})^{\rho_s^\vee}t(w_0\gamma)^{-1}\bigr)}
{\vartheta((k_0^{-1}k_{2a_0})^{\rho_s^\vee}t\bigr)}
\]
(note that $(k_0^{-1}k_{2a_0})^{\rho_s^\vee}=1$
in the reduced case). 
The prefactor $\widehat{\mathcal{W}}_\eta(t,\gamma)$   
satisfies the same functional equations as function
of $t\in T$ as the selfdual prefactor $\mathcal{W}(t,\gamma)$,
\[
\widehat{\mathcal{W}}_\eta(q^\lambda t,\gamma)=
\gamma_0^\lambda\gamma^{w_0\lambda}\widehat{\mathcal{W}}_\eta(t,
\gamma)\qquad \forall\,\lambda\in\Lambda.
\]
\begin{cor}
Let $\gamma\in T$ such that $\mathcal{L}_q(\gamma)\not=0$.
The monic basic Harish-Chandra series $\widehat{\Phi}_\eta(\cdot,\gamma)$ 
satisfies the Macdonald $q$-difference equations
\begin{equation}\label{Macdonaldqdiff}
D_p\widehat{\Phi}_\eta(\cdot,\gamma)=p(\gamma^{-1})\widehat{\Phi}_\eta(\cdot,
\gamma)\qquad \forall\,p\in\mathbb{C}[T]^{W_0}
\end{equation}
and has, for $t\in B_\epsilon^{-1}$ with $\epsilon>0$ sufficiently small,
a convergent series expansion
\[
\widehat{\Phi}_\eta(t,\gamma)=\widehat{\mathcal{W}}_\eta(t,\gamma)
\sum_{\mu\in Q_+}\Gamma_\mu(\gamma)t^{-\mu}
\]
where $\Gamma_\mu(\gamma):=
\phi(\widehat{\Upsilon}_\mu(\gamma))$ (in particular,
$\Gamma_0(\gamma)=1$). The series converges normally
for $t$ in compacta of $B_\epsilon^{-1}$. 
\end{cor}
Since
\begin{equation}\label{good}
\widehat{\mathcal{W}}_\eta(q^\lambda\eta\gamma_{0,d},\gamma)=
\gamma_0^\lambda\gamma^{w_0\lambda}\qquad \forall\,\lambda\in\Lambda,
\end{equation}
the monic basic Harish-Chandra series
$\widehat{\Phi}_\eta(\cdot,\gamma)$ is the natural normalization
of the basic Harish-Chandra series when restricting the
Macdonald $q$-difference equations \eqref{Macdonaldqdiff}
to functions on the $q$-lattice $\eta\gamma_{0,d}q^\Lambda$.
\begin{prop}\label{HCpol} Fix $\lambda\in\Lambda^-$.
For generic values of the multiplicity function $k$
we have
\begin{equation}\label{evaluations}
\widehat{\Phi}_\eta(t,\gamma_\lambda)=(\eta\gamma_{0,d})^{-w_0\lambda}
P_\lambda^+(t).
\end{equation}
\end{prop}
\begin{proof}
Fix $\lambda\in\Lambda^-$. Note that 
\[
\widehat{\mathcal{W}}(t,\gamma_\lambda)=
q^{\frac{-|\lambda|^2}{2}}k_0^{(\lambda,\rho_s^\vee)}
k_{2a_0}^{-(\lambda,\rho_s^\vee)}t^{w_0\lambda},
\]
hence for $t\in B_\epsilon^{-1}$ with $\epsilon>0$ sufficiently small,
\begin{equation}\label{expansionmon}
\widehat{\Phi}_\eta(t,\gamma_\lambda)=
\sum_{\mu\in Q_+}d_\mu t^{w_0\lambda-\mu}
\end{equation}
as normally convergent series for $t$ in compacta of $B_\epsilon^{-1}$, 
with leading coefficient
\begin{equation}\label{lt}
d_0=(\eta\gamma_{0,d})^{-w_0\lambda}
\end{equation}
(this requires $\mathcal{L}_q(\gamma_\lambda)\not=0$, which we impose
as one of the genericity conditions on the multiplicity function).
Since $k$ is generic, this 
characterizes $\widehat{\Phi}(\cdot,\gamma_\lambda)$ within the class
of formal power series $f\in\mathbb{C}[[X^{-\alpha_i}]]X^{w_0\lambda}$ 
satisfying the eigenvalue equations
\[
\bigl(D_pf\bigr)(t)=
p(\gamma_\lambda^{-1})f(t),\qquad \forall\,
p\in\mathbb{C}[T]^{W_0}
\]
(cf., e.g., \cite[Thm. 4.6]{LS}). The result now follows since
$f(t)=d_0P_\lambda^+(t)$ satisfies the same characterizing properties.
\end{proof}
\begin{rema}\label{consequenceremark}
The explicit evaluation formula \cite[\S 3.3.2]{C}
for the symmetric Macdonald-Koornwinder
polynomial $P_\lambda^+(\gamma_{0,d})=
P_\lambda^+(w_0\gamma_{0,d})=P_\lambda^+(\gamma_{0,d}^{-1})$
($\lambda\in\Lambda^-$) 
can be derived from Proposition \ref{HCpol} and
the fundamental properties of the selfdual basic Harish-Chandra series
\begin{equation}\label{relationPhi}
\Phi(t,\gamma)=\frac{\mathcal{W}(t,\gamma)}
{\widehat{\mathcal{W}}_\eta(t,\gamma)}
\frac{\mathcal{L}_q(\gamma)}
{\mathcal{S}_{k^d,q}(\gamma)\sum_{w\in W_0}k_w^2}\widehat{\Phi}_\eta(t,\gamma)
\end{equation}
as follows. By a direct computation using
Proposition \ref{HCpol}, 
\begin{equation}\label{l1}
\Phi(\gamma_{\mu,d}^{-1},\gamma_\lambda;k,q)=
\frac{\mathcal{W}(\gamma_{0,d}^{-1},\gamma_0;k,q)}{\sum_{w\in W_0}k_w^2}
\gamma_{0,d}^{-\lambda}\frac{\mathcal{L}_q(\gamma_\lambda)}
{\mathcal{S}_{k^d,q}(\gamma_\lambda)}P_\lambda^+(\gamma_{\mu,d}^{-1};k,q)
\end{equation}
for $\lambda,\mu\in\Lambda^-$.
By the selfduality of $\Phi$ and of $\mathcal{W}(\cdot,\cdot)$ 
it is also equal to
\begin{equation}\label{l2}
\Phi(\gamma_\lambda^{-1},\gamma_{\mu,d};k^d,q)=
\frac{\mathcal{W}(\gamma_{0,d}^{-1},\gamma_0;k,q)}{\sum_{w\in W_0}k_w^2}
\gamma_0^{-\mu}\frac{\mathcal{L}_q(\gamma_{\mu,d})}
{\mathcal{S}_{k,q}(\gamma_{\mu,d})}P_\mu^+(\gamma_\lambda^{-1};k^d,q).
\end{equation}
Setting $\lambda=\mu=0$ we get
\[
\frac{\mathcal{L}_q(\gamma_0)}{\mathcal{S}_{k^d,q}(\gamma_0)}=
\frac{\mathcal{L}_q(\gamma_{0,d})}{\mathcal{S}_{k,q}(\gamma_{0,d})}.
\]
Setting $\mu=0$ we then get the evaluation formula
\[
P_\lambda^+(\gamma_{0,d}^{-1})=
\gamma_{0,d}^\lambda\frac{\mathcal{L}_q(\gamma_0)}
{\mathcal{S}_{k^d,q}(\gamma_0)}\frac{\mathcal{S}_{k^d,q}(\gamma_\lambda)}
{\mathcal{L}_q(\gamma_\lambda)}.
\]
Returning to \eqref{l1} and \eqref{l2} it yields the well known
selfduality
\[
E_+(\gamma_\lambda;\gamma_{\mu,d}^{-1};k,q)=
E_+(\gamma_{\mu,d},\gamma_\lambda^{-1};k^d,q)\qquad 
\forall \lambda,\mu\in\Lambda^-
\]
of the symmetric Macdonald-Koornwinder polynomials. Using
$E_+(\gamma_\lambda;t)=E_+(\gamma_{-w_0\lambda};t^{-1})$ 
and $\gamma_{\lambda}^{-1}=w_0\gamma_{-w_0\lambda}$ for $\lambda\in\Lambda^-$
the selfduality can be rewritten as
\[
E_+(\gamma_\lambda;\gamma_{\mu,d};k,q)=E_+(\gamma_{\mu,d};\gamma_{\lambda};k^d,q)
\qquad\forall\,\lambda,\mu\in\Lambda^-.
\]
\end{rema}

\section{The $c$-function expansion}\label{c}

The existence of an expansion of the basic hypergeometric function
$\mathcal{E}_+$ in terms of basic Harish-Chandra series follows
now readily:

\begin{prop}\label{existence}
$\{\Phi(\cdot,w\cdot)\}_{w\in W_0}$ is a $\mathbb{F}$-basis of
the subspace $\phi(\mathcal{K})$ of $\mathcal{U}$. Hence
there exists a unique $\mathfrak{c}(\cdot,\cdot)=
\mathfrak{c}(\cdot,\cdot;k,q)\in\mathbb{F}$
such that
\begin{equation}\label{star}
\mathcal{E}_+(t,\gamma)=\sum_{w\in W_0}\mathfrak{c}(t,w\gamma)\Phi(t,w\gamma).
\end{equation}
\end{prop}
\begin{proof}
Since $\phi:\mathcal{K}\rightarrow\mathcal{U}$ is $W_0\times W_0$-equivariant,
we have
\[
\phi(\nabla(1,w)F)=\Phi(\cdot,w^{-1}\cdot),\qquad w\in W_0.
\]
The first statement then follows from Theorem \ref{properties}{\bf (ii)}.
By Corollary \ref{imagephi} we have 
\[
\mathcal{E}_+(t,\gamma)=\sum_{w\in W_0}\mathfrak{c}_w(t,\gamma)\Phi(t,w\gamma)
\]
in $\phi(\mathcal{K})\subset\mathcal{U}$ for unique 
$\mathfrak{c}_w\in\mathbb{F}$
($w\in W_0$). Since $\mathcal{E}_+$ is $W_0\times W_0$-invariant,
$\mathfrak{c}_w(t,\gamma)=\mathfrak{c}_1(t,w\gamma)$ for $w\in W_0$.
\end{proof}
We are now going to derive an explicit expression of
the expansion coefficient 
$\mathfrak{c}\in\mathbb{F}$ in terms of theta functions.
As a first step we will single out the $t$-dependence.
The following preliminary lemma 
is closely related to \cite[Thm. 4.1 (i)]{CWhit} (reduced case).

Set
\[
\widetilde{\rho}:=\varpi_1+\cdots+\varpi_n\in\Lambda^+.
\]
\begin{lem}\label{hlambda}
Fix generic $\gamma\in T$ with $|\gamma^{-\alpha_i}|\leq 1$ 
for $1\leq i\leq n$.
For $\lambda\in\Lambda^-$ define 
$h_\lambda\in\mathcal{M}(T)$ by 
\[
h_\lambda(t):=\gamma_0^{-\lambda}\gamma^{-w_0\lambda}
\left(\prod_{\alpha\in R_{0,s}^+}\frac{\bigl(-q_\theta k_0k_{2a_0}^{-1}t^{-\alpha};
q_\theta^2\bigr)_{-(\lambda,\alpha^\vee)/2}}
{\bigl(-q_\theta k_{2a_0}k_0^{-1}t^{-\alpha};q_\theta^2
\bigr)_{-(\lambda,\alpha^\vee)/2}}\right)\frac{\mathcal{E}_+(q^\lambda t,\gamma)}
{G_{k^\tau,q}(t)G_{k^{d\tau},q}(\gamma)}
\] 
(in the reduced case we have $k_0=k_{2a_0}$, hence in this case
the product over $R_{0,s}^+$ is one; in the nonreduced case
$(\lambda,\alpha^\vee)$ is even for all 
$\alpha\in R_{0,s}^+$). 
Then $h_\lambda$ is holomorphic on $T$ and  
\[
\lim_{r\rightarrow\infty}h_{-r\widetilde{\rho}}(t)
\]
converges to a holomorphic function $h_{-\infty}(t)$ in $t\in T$.
\end{lem}
\begin{proof}
Observe that
\[
\left(\prod_{\alpha\in R_{0,s}^+}\frac{\bigl(-q_\theta k_0k_{2a_0}^{-1}t^{-\alpha};
q_\theta^2\bigr)_{-(\lambda,\alpha^\vee)/2}}
{\bigl(-q_\theta k_{2a_0}k_0^{-1}t^{-\alpha};q_\theta^2
\bigr)_{-(\lambda,\alpha^\vee)/2}}\right)\frac{G_{k^\tau,q}(q^\lambda t)}
{G_{k^\tau,q}(t)}
\]
is a regular function in $t\in T$, and
$G_{k^\tau,q}(q^\lambda t)^{-1}G_{k^{d\tau},q}(\gamma)^{-1}\mathcal{E}_+(q^\lambda 
t,\gamma)$ is holomorphic in $(t,\gamma)\in T\times T$. 
Hence $h_\lambda(t)$ is holomorphic.
It remains to show that the $h_\lambda(t)$ 
($\lambda\in\Lambda^-$) are uniformly bounded for $t$ in compacta of $T$. 
Without loss
of generality it suffices to prove uniform boundedness for $t$ in compacta
of $B_\epsilon^{-1}$
for sufficiently small $\epsilon>0$.

Set 
\[
\mathcal{F}(t,\gamma):=\frac{\mathcal{E}(t,\gamma)}{G_{k^\tau,q}(t)
G_{k^{d\tau},q}(\gamma)},
\]  
which is the holomorphic part of the nonsymmetric basic hypergeometric
function $\mathcal{E}$.
For $w\in W_0$ let $v_w^*$ be the $\mathbb{K}$-linear functional on
$\mathbb{K}\otimes M$ mapping $v_{w^\prime}$ to $\delta_{w,w^\prime}$.
Recall from Theorem \ref{relspaces} that
\[
\mathcal{E}_+=\pi^t(C_+)\mathcal{E}=\phi(\psi\mathcal{E})
\]
and $\psi\mathcal{E}\in\mathcal{K}$. Hence
\[
(\psi\mathcal{E})(q^\lambda t,\gamma)=C_{(\tau(\lambda),1)}(q^\lambda t,\gamma)
(\psi\mathcal{E})(t,\gamma),
\] 
so that
\begin{equation*}
\begin{split}
h_\lambda(t)&=\frac{k_{w_0}^{-1}}{\sum_{w\in W_0}k_w^{-2}}
\prod_{\alpha\in R_{0,s}^+}\frac{\bigl(-q_\theta k_0k_{2a_0}^{-1}t^{-\alpha};
q_\theta^2\bigr)_{-(\lambda,\alpha^\vee)/2}}
{\bigl(-q_\theta k_{2a_0}k_0^{-1}t^{-\alpha};q_\theta^2
\bigr)_{-(\lambda,\alpha^\vee)/2}}\\
&\times
\sum_{w,w^\prime\in W_0}k_w^{-1}
\bigl(\pi^t(T_{w^\prime w_0})\mathcal{F}\bigr)(t,\gamma)
v_w^*\bigl(\gamma_0^{-\lambda}\gamma^{-w_0\lambda}
C_{(\tau(\lambda),1)}(q^\lambda t,\gamma)
v_{w^\prime}\bigr).
\end{split}
\end{equation*}
It thus suffices to give bounds for
$v_w^*(D_\lambda(t)v_{w^\prime})$ ($\lambda\in\Lambda^-$),
uniform for $t$ in compacta of $B_\epsilon^{-1}$,
where
\[D_\lambda(t):=\gamma_0^{-\lambda}\gamma^{-w_0\lambda}
C_{(\tau(\lambda),1)}(q^\lambda t,\gamma).
\]
Recall from Theorem \ref{AF} 
the asymptotically free solution $F(\cdot,\cdot)=\mathcal{W}(\cdot,\cdot)
H(\cdot,\cdot)$ of the bispectral quantum KZ equations. Then
\[
\gamma^{w_0\lambda-ww_0\lambda}D_\lambda(t)H_w(t)=H_w(q^\lambda t),\qquad w\in W_0
\]
with 
\[
H_w(t):=\bigl(\nabla(e,w)H\bigr)(t,\gamma)=
C_{(1,w)}(t,\gamma)H(t,w^{-1}\gamma)
\]
(note that the $C_{(1,w)}(t,\gamma)$ ($w\in W_0$) do not depend on
$t$). Furthermore, writing
\[
H_{w^\prime}(t)=\sum_{w\in W_0}a_{w}^{w^\prime}(t)v_{w},
\]
the matrix $A(t):=\bigl(a_w^{w^\prime}(t)\bigr)_{w,w^\prime\in W_0}$
is invertible (cf. the proof of \cite[Lem. 5.12]{vMSt}) and both
$A(t)$ and $A(t)^{-1}$ are uniformly bounded for $t\in B_\epsilon^{-1}$.
 Writing $N(t,\lambda)$ for the
matrix $\bigl(v_w^*(D_\lambda(t)v_{w^\prime})\bigr)_{w,w^\prime\in W_0}$
we conclude that
\[
N(t,\lambda)=A(q^\lambda t)M(\lambda)A(t)^{-1},
\]
where $M(\lambda)$ is the diagonal matrix 
$\bigl(\delta_{w,w^\prime}\gamma^{-w_0\lambda+ww_0\lambda}\bigr)_{w,w^\prime\in W_0}$.
The matrix coefficients of $M(\lambda)$ are 
bounded as function of $\lambda\in\Lambda^-$ since
$|\gamma^{-\alpha_i}|\leq 1$ for all $i$.
This implies the required boundedness conditions for the
matrix coefficients of $N(t,\lambda)$.
\end{proof}

Set
\begin{equation}\label{cvartheta}
\mathfrak{c}^\vartheta(t,\gamma;k,q):=
\frac{\vartheta(\gamma_0^{-1}(k_0^{-1}k_{2a_0})^{\rho_s^\vee}t(w_0\gamma)^{-1})
\vartheta(\gamma_0t)\vartheta(\gamma_{0,d}^{-1}\gamma)
\vartheta((k_0^{-1}k_{2a_0})^{\rho_s^\vee}\gamma_{0,d})}
{\vartheta((k_{2a_0}k_{2\theta}^{-1})^{\rho_s^\vee}\gamma)
\vartheta(t(w_0\gamma)^{-1})\vartheta((k_0^{-1}k_{2a_0})^{\rho_s^\vee}t)}
\end{equation}
Observe that $\mathfrak{c}^\vartheta$ satisfies the functional equations
\begin{equation*}
\begin{split}
\mathfrak{c}^\vartheta(q^\lambda t,\gamma)&=\mathfrak{c}^\vartheta(t,\gamma)\\
\mathfrak{c}^\vartheta(t,q^\lambda\gamma)&=
\gamma_{0,d}^{2\lambda}\mathfrak{c}^\vartheta(t,\gamma)
\end{split}
\end{equation*}
for $\lambda\in\Lambda$.
Since $\gamma_0^{-1}\gamma_{0,d}=(k_{0}k_{2\theta}^{-1})^{\rho_s^\vee}$
we in addition have
\begin{equation}\label{normalizationc}
\mathfrak{c}^\vartheta(\gamma_{0,d},\gamma)=
\frac{1}{\mathcal{W}(\gamma_{0,d},\gamma)}.
\end{equation} 
\begin{cor}\label{decompcor}
The expansion coefficient 
$\mathfrak{c}\in\mathbb{F}$ in \eqref{star} is of the form
\begin{equation}\label{fact}
\mathfrak{c}(t,\gamma)=\mathfrak{c}^\vartheta(t,\gamma)\mathfrak{c}^\theta(\gamma)
\end{equation}
for a unique 
$\mathfrak{c}^\theta(\cdot)=
\mathfrak{c}^\theta(\cdot;k,q)\in\mathcal{M}(T)$ satisfying
the functional equations
\begin{equation}\label{fr}
\mathfrak{c}^\theta(q^\lambda\gamma)=
\gamma_{0,d}^{-2\lambda}\mathfrak{c}^\theta(\gamma)\qquad
\forall\,\lambda\in\Lambda.
\end{equation}
\end{cor}
\begin{proof}
In view of the functional equations of 
$\mathfrak{c}^\vartheta(t,\gamma)$ in $\gamma$ it
suffices to prove the factorization for generic $t,\gamma\in T$ satisfying 
$|\gamma^{-\alpha_i}|<1$ for all $1\leq i\leq n$. Since
\begin{equation}\label{useful}
\frac{1}{G_{k^\tau,q}(t)}
\prod_{\alpha\in R_{0,s}^+}\frac{\bigl(-q_\theta k_0k_{2a_0}^{-1}t^{-\alpha};
q_\theta^2\bigr)_{\infty}}{\bigl(-q_\theta k_{2a_0}k_0^{-1}t^{-\alpha};
q_\theta^2\bigr)_{\infty}}=\vartheta\bigl((k_0^{-1}k_{2a_0})^{\rho_s^\vee}t\bigr)
\end{equation}
(which is trivial in the reduced case and follows 
by a direct computation in the nonreduced case) we have
\begin{equation}\label{altexpression}
\begin{split}
h_{-\infty}(t)&=\frac{\vartheta((k_0^{-1}k_{2a_0})^{\rho_s^\vee}t)}
{G_{k^{d\tau},q}(\gamma)}
\sum_{w\in W_0}\mathfrak{c}(t,w\gamma)\mathcal{W}(t,w\gamma)
\lim_{r\rightarrow\infty}\gamma^{r(w_0\widetilde{\rho}-w^{-1}w_0\widetilde{\rho})}
(\phi H)(q^{-r\widetilde{\rho}}t,w\gamma)\\
&=\frac{\vartheta((k_0^{-1}k_{2a_0})^{\rho_s^\vee}t)
\mathfrak{c}(t,\gamma)\mathcal{W}(t,\gamma)\mathcal{L}_q(\gamma)}
{G_{k^{d\tau},q}(\gamma)\mathcal{S}_{k^d,q}(\gamma)
\sum_{w\in W_0}k_w^2}
\end{split}
\end{equation}
by Theorem \ref{properties},
Proposition \ref{existence}, \eqref{frW}, \eqref{decPhi} 
and the assumption that $|\gamma^{-\alpha_i}|<1$ for all 
$1\leq i\leq n$. 
It follows from this expression that the holomorphic function $h_{-\infty}$ 
satisfies
\[
h_{-\infty}(q^\lambda t)=q^{-\frac{|\lambda|^2}{2}}
\bigl(\gamma_0^{-1}(k_0^{-1}k_{2a_0})^{\rho_s^\vee}t(w_0\gamma)^{-1}\bigr)^{-\lambda}
h_{-\infty}(t)\quad \forall\,\lambda\in\Lambda.
\]
Consequently 
\[
h_{-\infty}(t)=\textup{cst}\,\vartheta\bigl(\gamma_0^{-1}
(k_0^{-1}k_{2a_0})^{\rho_s^\vee}
t(w_0\gamma)^{-1}\bigr)
\]
for some $\textup{cst}\in\mathbb{C}$ independent of $t\in T$.
Combined with the second line of \eqref{altexpression} one obtains
the desired result.
\end{proof}
The factor 
$\mathfrak{c}^\vartheta(t,\gamma)$ is highly dependent on our specific choice of
(selfdual, meromorphic) prefactor $\mathcal{W}$
in the selfdual basic Harish-Chandra
series. We will see later that this factor simplifies
when considering the expansion of the basic
hypergeometric function
in terms of the monic basic Harish-Chandra series. In particular it will
no longer depend on the first torus variable $t\in T$.

The next step is to compute $\mathfrak{c}^\theta(\gamma)=
\mathfrak{c}^\theta(\gamma;k,q)$ 
explicitly. We will obtain an expression 
in terms of the Jacobi theta function 
$\theta(\cdot;q)$.
Recall that
\begin{equation}\label{expl1}
\frac{\mathcal{S}_{k,q}(t)}{\mathcal{L}_q(t)}=
\prod_{\stackrel{\alpha\in R_0^+,}{r\in\mathbb{Z}_{>0}}}
c_{\alpha+r\frac{|\alpha|^2}{2}c}(t;k^{-1},q).
\end{equation}
We define a closely related meromorphic function $c(\cdot)=
c_{k,q}(\cdot)\in\mathcal{M}(T)$
by
\begin{equation}\label{expl2}
c(t):=\prod_{\stackrel{\alpha\in R_0^+,}{r\in\mathbb{Z}_{\geq 0}}}c_{-\alpha+
r\frac{|\alpha|^2}{2}c}(t;k,q).
\end{equation}
An explicit computation yields expressions of
both \eqref{expl1} and \eqref{expl2} in terms of
of $q$-shifted factorials. For $c(t)$ it reads
\begin{equation}\label{Vr}
c(t)=\prod_{\alpha\in R_0^+}\frac{\bigl(k_\alpha^2t^{-\alpha};q_\alpha\bigr)_{\infty}}
{\bigl(t^{-\alpha};q_\alpha\bigr)_{\infty}}
\end{equation}
in the reduced case and
\begin{equation}\label{Vnr}
c(t)=\prod_{\alpha\in R_{0,l}^+}\frac{\bigl(k_\vartheta^2t^{-\alpha};q_\vartheta
\bigr)_{\infty}}{\bigl(t^{-\alpha};q_\vartheta\bigr)_{\infty}}
\prod_{\beta\in R_{0,s}^+}\frac{\bigl(at^{-\beta}, bt^{-\beta}, ct^{-\beta},
dt^{-\beta};q_\theta^2\bigr)_{\infty}}{\bigl(t^{-2\beta};q_\theta^2\bigr)_{\infty}}
\end{equation}
in the nonreduced case,
where $R_{0,l}^+\subset R_0^+$ is the subset of positive long roots and
$\{a,b,c,d\}$ are given by \eqref{AWpar}. The product formula of 
$c_{k^d,q}(\gamma)$ in the reduced case is the $q$-analog 
of the Gindikin-Karpelevic \cite{GK} product formula 
of the Harish-Chandra $c$-function as well as of its extension 
to the Heckman-Opdam theory (see \cite[Part I, Def. 3.4.2]{HS}). 

Taking the product 
\[
\frac{\mathcal{S}_{k,q}(t)c_{k,q}(t)}
{\mathcal{L}_q(t)}
\]
of \eqref{expl1} and \eqref{expl2}, the $q$-shifted factorials can be 
pairwise combined to yield the following 
explicit expression in terms of Jacobi's
theta function $\theta(\cdot;q)$.
\begin{lem}\label{thetaexpr}
{\bf (i)} In the reduced case,
\[
\frac{\mathcal{S}_{k,q}(t)c_{k,q}(t)}{\mathcal{L}_q(t)}=
\prod_{\alpha\in R_0^+}\frac{\theta(k_\alpha^2t^{-\alpha};q_\alpha)}
{\theta(t^{-\alpha};q_\alpha)}.
\]
{\bf (ii)} In the nonreduced case,
\[
\frac{\mathcal{S}_{k,q}(t)c_{k,q}(t)}{\mathcal{L}_q(t)}=
\prod_{\alpha\in R_{0,l}^+}\frac{\theta(k_\vartheta^2t^{-\alpha};q_\vartheta)}
{\theta(t^{-\alpha};q_\vartheta)}
\prod_{\beta\in R_{0,s}^+}\frac{\theta(at^{-\beta};q_\theta^2)
\theta(bt^{-\beta};q_\theta^2)\theta(ct^{-\beta};q_\theta^2)
\theta(dt^{-\beta};q_\theta^2)}{\bigl(q_\theta^2;q_\theta^2)_{\infty}^3
\theta(t^{-2\beta};q_\theta^2)}.
\]
\end{lem}
In Subsection \ref{Abqkz} we have seen that 
\[\frac{\mathcal{L}_q(\gamma)}{\mathcal{S}_{k^d,q}(\gamma)
\sum_{w\in W_0}k_w^2}
\]
is the leading coefficient of the power series expansion of
$(\phi H)(t,\gamma)$ in the variables 
$t^{-\alpha_i}$ ($1\leq i\leq n$). On the other hand
it is
closely related to the evaluation formula for the symmetric 
Macdonald-Koornwinder
polynomials, see Remark \ref{consequenceremark}.
In the next proposition we show that the meromorphic function
$c(t)$ governs the asymptotics of the symmetric Macdonald-Koornwinder
polynomial.
In the reduced case it
is due to Cherednik \cite[Lemma 4.3]{CWhit} (for the rank one case
see, e.g., \cite{IW}).
\begin{prop}\label{poll}
For $\epsilon>0$ sufficiently small,
\begin{equation}\label{toto}
\lim_{r\rightarrow\infty}\gamma_{0,d}^{r\widetilde{\rho}}t^{rw_0\widetilde{\rho}}
E_+(\gamma_{-r\widetilde{\rho}};t)=\frac{c(t)}{c(\gamma_{0,d})},
\end{equation}
normally converging for $t$ in compacta of $B_\epsilon^{-1}$. 
\end{prop}
\begin{proof}
The proof in the reduced case (see \cite{CWhit})
consists of analyzing the gauged
Macdonald $q$-difference equations $t^{rw_0\widetilde{\rho}}\circ D_p\circ
t^{-rw_0\widetilde{\rho}}$ ($p\in\mathbb{C}[T]^{W_0}$)
in the limit $r\rightarrow\infty$ and observing that the left and 
right hand side of \eqref{toto} are the (up to normalization)
unique solution of the resulting residual $q$-difference equations that
have a series expansion in $t^{-\mu}$ ($\mu\in Q_+$), normally converging
for $t$ in compacta of $B_\epsilon^{-1}$.
This proof can be straightforwardly extended to the nonreduced case.
\end{proof}

\begin{thm}\label{mainresult}
We have
\[
\mathcal{E}_+(t,\gamma;k,q)=\sum_{w\in W_0}\mathfrak{c}(t,w\gamma;k,q)
\Phi(t,w\gamma;k,q)
\]
with $\mathfrak{c}(\cdot,\cdot)=
\mathfrak{c}(\cdot,\cdot;k,q)\in\mathbb{F}$ given by
\begin{equation}\label{explicitc}
\mathfrak{c}(t,\gamma;k,q)=
\mathfrak{c}^\vartheta(t,\gamma;k,q)\mathfrak{c}^\theta(\gamma;k,q)
\end{equation}
where $\mathfrak{c}^\vartheta(\cdot,\cdot)=\mathfrak{c}^\vartheta(\cdot,\cdot;k,q)
\in\mathbb{K}$ 
is given by \eqref{cvartheta} and 
$\mathfrak{c}^\theta(\cdot)=\mathfrak{c}^\theta(\cdot;k,q)
\in\mathcal{M}(T)$ is given by
\begin{equation}\label{ctheta}
\mathfrak{c}^\theta(\gamma;k,q)=
\frac{\mathcal{S}_{k^d,q}(\gamma)c_{k^d,q}(\gamma)}
{\mathcal{L}_q(\gamma)}\frac{\sum_{w\in W_0}k_w^2}{c_{k^d,q}(\gamma_0)}.
\end{equation}
\end{thm}
In view of Lemma \ref{thetaexpr}, formula \eqref{ctheta}
provides an explicit
expression of $\mathfrak{c}^\theta(\gamma)$ as product of Jacobi theta
functions.
\begin{proof}
Using Lemma \ref{thetaexpr} it is easy to check that
the right hand side of \eqref{ctheta} satisfies the
functional equations \eqref{fr}.
Hence it suffices to prove the explicit expression \eqref{ctheta}
of $\mathfrak{c}^\theta(\gamma)$ for generic $\gamma\in T$ such that
$|\gamma^{-\alpha_i}|$ is sufficiently small for all
$1\leq i\leq n$.
We fix such $\gamma$ in the remainder of the proof.
Recall the associated holomorphic function $h_\lambda(t)$ in $t\in T$
from Lemma \ref{hlambda}. By Theorem \ref{QSP}, 
using \eqref{useful}
for the first equality and Proposition \ref{poll} for the second
equality,
\begin{equation*}
\begin{split}
h_{-\infty}(\gamma_{0,d})&=
\frac{\vartheta\bigl((k_0^{-1}k_{2a_0})^{\rho_s^\vee}\gamma_{0,d}\bigr)}
{G_{k^{d\tau},q}(\gamma)}
\lim_{r\rightarrow\infty}
\gamma_0^{r\widetilde{\rho}}\gamma^{rw_0\widetilde{\rho}}
E_+(\gamma_{-r\widetilde{\rho},d};\gamma;k^d,q)\\
&=\frac{\vartheta\bigl((k_0^{-1}k_{2a_0})^{\rho_s^\vee}\gamma_{0,d}\bigr)
c_{k^d,q}(\gamma)}
{c_{k^d,q}(\gamma_0)G_{k^{d\tau},q}(\gamma)}.
\end{split}
\end{equation*}
On the other hand, by \eqref{altexpression}, Corollary
\ref{decompcor} and \eqref{normalizationc},
\[
h_{-\infty}(\gamma_{0,d})=
\frac{\vartheta((k_0^{-1}k_{2a_0})^{\rho_s^\vee}\gamma_{0,d})
\mathcal{L}_q(\gamma)}
{G_{k^{d\tau},q}(\gamma)\mathcal{S}_{k^d,q}(\gamma)\sum_{w\in W_0}k_w^2}
\mathfrak{c}^\theta(\gamma).
\]
Combining these two formulas 
yields the desired expression for $\mathfrak{c}^\theta(\gamma)$.
\end{proof}
A direct computation using \eqref{relationPhi} gives now
the following $c$-function expansion
of the basic hypergeometric function $\mathcal{E}_+$ in terms
of the monic basic Harish-Chandra series $\widehat{\Phi}_\eta$.
\begin{cor}\label{monicc}
For generic $\eta\in T$ we have
\begin{equation}\label{moniccformula}
\mathcal{E}_+(t,\gamma;k,q)=\widehat{c}_\eta(\gamma_0;k,q)^{-1}
\sum_{w\in W_0}\widehat{c}_\eta(w\gamma;k,q)
\widehat{\Phi}_\eta(t,w\gamma;k,q)
\end{equation}
with $\widehat{c}_\eta(\cdot)=\widehat{c}_\eta(\cdot;k,q)\in\mathcal{M}(T)$ 
explicitly given by
\[
\widehat{c}_\eta(\gamma;k,q)=\frac{\vartheta\bigl((w_0\eta)^{-1}
(k_{2a_0}k_{2\theta}^{-1})^{\rho_s^\vee}\gamma\bigr)}
{\vartheta\bigl((k_{2a_0}k_{2\theta}^{-1})^{\rho_s^\vee}\gamma\bigr)}
c_{k^d,q}(\gamma).
\]
\end{cor}
In the rank one case 
the $c$-function expansion of $\mathcal{E}_+$ was established
by direct computations in \cite{CO} ($\textup{GL}_2$ case)
and in \cite{KS,Stexp} (nonreduced rank one case). We return to the rank
one case and establish the connections to 
basic hypergeometric series in Section \ref{SC}.

In the reduced case, by \eqref{Vr} the coefficient $\widehat{c}_\eta$
explicitly reads
\[
\widehat{c}_\eta(\gamma)=\frac{\vartheta((w_0\eta)^{-1}\gamma)}
{\vartheta(\gamma)}
\prod_{\alpha\in R_0^+}\frac{\bigl(k_\alpha^2\gamma^{-\alpha};q_\alpha\bigr)_{\infty}}
{\bigl(\gamma^{-\alpha};q_\alpha\bigr)_{\infty}}.
\]
By \eqref{Vnr} an explicit
expression of the monic $c$-function
$\widehat{c}_\eta$ 
can also be given in the nonreduced case, 
see \eqref{cnr} for the resulting expression.

Note that the formula for $\widehat{c}_\eta(\gamma)$
simplifies for $\eta=1$ to
\[
\widehat{c}_1(\gamma;k,q)=c_{k^d,q}(\gamma).
\]
As remarked in the introduction this shows 
that $\mathcal{E}_+$ formally is a $q$-analog
of the Heckman-Opdam \cite{HO,HS,O} hypergeometric function.

The $\eta$-dependence is expected to be of importance in
the applications to harmonic analysis on noncompact quantum groups.
In \cite{KSsu11,KS} a selfdual
spherical Fourier transform on
the quantum $\textup{SU}(1,1)$ quantum group was defined and studied
whose Fourier kernel is given by the nonreduced rank one basic hypergeometric
function $\mathcal{E}_+$, which is the Askey-Wilson function
from \cite{KS} (see Subsection \ref{specialnr}). 
The Fourier transform and the Plancherel
measure were defined in terms of the Plancherel density function
\begin{equation}\label{mueta}
\mu_\eta(\gamma)=\frac{1}{\widehat{c}_\eta(\gamma)\widehat{c}_\eta(\gamma^{-1})}.
\end{equation}
The extra theta-factors compared to the familiar
Macdonald density 
\[
\mu_1(\gamma)=\frac{1}{c_{k^d,q}(\gamma)c_{k^d,q}(\gamma^{-1})}
\]
lead to an infinite set of discrete mass points in
the associated (inverse of the) Fourier transform.
In its interpretation
as spherical Fourier transform these mass points account for  
the contributions of the strange series representations to the 
Plancherel measure (the stranges series is a series
of irreducible unitary representations
of the quantized universal enveloping algebra which vanishes
in the limit $q\rightarrow 1$, see \cite{Metal}).
Crucial ingredients for obtaining the Plancherel and inversion formulas
are the explicit $c$-function expansion 
and the selfduality of the Askey-Wilson function $\mathcal{E}_+$. 
The generalization of these results to arbitrary root systems is not
known.

\section{Special cases and applications}\label{SC}
\subsection{Asymptotics of symmetric Macdonald-Koornwinder polynomials}
\label{abcd}
As a consequence of the $c$-function expansion we can establish pointwise 
asymptotics of the symmetric Mac\-do\-nald-Koornwinder polynomials
when the degree tends to infinity. The $L^2$-asymptotics was established
in \cite{RuL} for $\textup{GL}_m$, \cite{vD1} for the reduced case and
\cite{vD2} for the nonreduced case (for 
the rank one cases see e.g. \cite{IW},
\cite[\S 7.4 \& \S 7.5]{GR} and references therein).

For $\lambda\in\Lambda^-$ set 
\[
m(\lambda):=\textup{max}((\lambda,\alpha_i)\,\, | \,\,
1\leq i\leq n)\in\mathbb{R}_{\leq 0}.
\]
\begin{cor}
Fix $t\in T$ such that $\mathcal{S}_{k,q}(wt)\not=0$ for all $w\in W_0$.
Then
\begin{equation}\label{asymp}
E_+(\gamma_\lambda;t;k,q)=\sum_{w\in W_0}
\frac{c_{k,q}(wt)}{c_{k,q}(\gamma_{0,d})}\gamma_{0,d}^\lambda t^{w^{-1}w_0\lambda}
+\mathcal{O}(q^{-m(\lambda)})
\end{equation}
as $m(\lambda)\rightarrow -\infty$.
\end{cor}
\begin{proof}
By the $c$-function expansion in selfdual form, 
Theorem \ref{QSP}{\bf (ii)-(iv)}, \eqref{relationPhi} 
and the expression $\widehat{\Phi}_\eta=
\widehat{\mathcal{W}}_\eta\phi(\widehat{H})$ 
we have for $\lambda\in\Lambda^-$,
\begin{equation*}
\begin{split}
E_+(\gamma_\lambda;t;k,q)&=\mathcal{E}_+(\gamma_\lambda,t;k^d,q)\\
&=\sum_{w\in W_0}\mathfrak{c}(\gamma_\lambda,wt;k^d,q)
\mathcal{W}(\gamma_\lambda,wt;k^d,q)\frac{\mathcal{L}_q(wt)}
{\mathcal{S}_{k,q}(wt)\sum_{v\in W_0}k_v^2}
(\phi\widehat{H})(\gamma_\lambda,wt;k^d,q).
\end{split}
\end{equation*}
The corollary now follows easily from 
the asymptotic series expansion in $\gamma_\lambda^{-\alpha}$
($\alpha\in Q_+$) of $(\phi\widehat{H})(\gamma_\lambda,wt,k^d;q)$,
together with \eqref{frW}, \eqref{normalizationc} and the explicit
expression \eqref{explicitc} of $\mathfrak{c}\in\mathbb{F}$. 
\end{proof}

\subsection{The nonreduced case}\label{specialnr}
We realize the root system $R_0\subset V_0=V$ of type
$\textup{B}_n$ as
$R_0=\{\pm\epsilon_i\}_{i=1}^n\cup\{\pm(\epsilon_i\pm\epsilon_j)\}_{1\leq 
i<j\leq n}$, with $\{\epsilon_i\}_{i=1}^n$ a fixed 
orthonormal basis of $V$.
We take as ordered basis
\[
\Delta_0=(\epsilon_1-\epsilon_2,\ldots,\epsilon_{n-1}-\epsilon_n,\epsilon_n)
\]
so that $R_{0,s}^+=\{\epsilon_i\}_{i=1}^n$, $R_{0,l}^+=
\{\epsilon_i\pm\epsilon_j\}_{1\leq i<j\leq n}$ and $\theta=\epsilon_1$,
$\vartheta=\epsilon_1+\epsilon_2$. We include $n=1$ 
as $R_0=\{\pm\epsilon_1\}$, the root system of type $\textup{A}_1$ 
(it amounts to omitting in the formulas below the
factors involving the long roots 
$\{\pm(\epsilon_i\pm\epsilon_j)\}_{i<j}$). We have $\Lambda=Q=
\bigoplus_{i=1}^n\mathbb{Z}\epsilon_i$.
We identify $T\simeq \bigl(\mathbb{C}^*\bigr)^n$, taking $t_i:=t^{\epsilon_i}$
as the coordinates. Note that $q_\theta=q^{\frac{1}{2}}$ and $q_\vartheta=q$.

The $q$-difference operator $D_p$ with respect to 
\[
p(t):=\sum_{i=1}^n(t_i+t_i^{-1})
\] 
was identified with Koornwinder's \cite{Ko} multivariable extension of
the Askey-Wilson second order $q$-difference operator by Noumi \cite{N}. 
It can most conveniently
be expressed in terms of the Askey-Wilson parameters
\[
\{a,b,c,d\}=\{k_\theta k_{2\theta}, -k_\theta k_{2\theta}^{-1},
q^{\frac{1}{2}}k_0k_{2a_0},-q^{\frac{1}{2}}k_0k_{2a_0}^{-1}\}
\]
and the dual Askey-Wilson parameters
\[
\{\widetilde{a},\widetilde{b},\widetilde{c},\widetilde{d}\}=
\{k_\theta k_0,-k_\theta k_0^{-1},q^{\frac{1}{2}}k_{2\theta}k_{2a_0},
-q^{\frac{1}{2}}k_{2\theta}k_{2a_0}^{-1}\}
\]
(which are the Askey-Wilson parameters associated to the 
multiplicity function $k^d$ dual to $k$) as
\[
D_p=\widetilde{a}^{-1}k_\vartheta^{2(1-n)}\bigl(\mathcal{D}+
\sum_{i=1}^n\bigl(\widetilde{a}^2k_\vartheta^{2(2n-i-1)}+k_\vartheta^{2(i-1)}\bigr)
\bigr)
\]
with
\begin{equation*}
\begin{split}
\mathcal{D}&=\sum_{i=1}^n\bigl(A_i(t)(\tau(-\epsilon_i)_q-1)+
A_i(t^{-1})(\tau(\epsilon_i)_q-1)\bigr),\\
A_i(t)&=\frac{(1-at_i)(1-bt_i)(1-ct_i)(1-dt_i)}{(1-t_i^2)(1-qt_i^2)}
\prod_{j\not=i}\frac{(1-k_\vartheta^2t_it_j)(1-k_\vartheta^2t_it_j^{-1})}
{(1-t_it_j)(1-t_it_j^{-1})}.
\end{split}
\end{equation*}
If follows that $P_\lambda^+$ ($\lambda\in\Lambda^-$) are the monic 
symmetric
Koornwinder \cite{Ko} polynomials 
and $E(\gamma_\lambda;\cdot)$ ($\lambda\in\Lambda^-$) are Sahi's \cite{Sa}
normalized symmetric Koornwinder polynomials. 

We now make the 
(monic version of) the $c$-function expansion of the associated
basic hypergeometric function $\mathcal{E}_+$ more explicit (see Corollary
\ref{monicc}). First note that $\rho_s^\vee=\sum_{i=1}^n\epsilon_i$
and that
\begin{equation*}
\begin{split}
\gamma_0^{-1}&=\bigl(\widetilde{a}k_\vartheta^{2(n-1)},\cdots,
\widetilde{a}k_\vartheta^2,\widetilde{a}),\\
\gamma_{0,d}^{-1}&=\bigl(ak_\vartheta^{2(n-1)},\ldots,ak_\vartheta^2,a).
\end{split}
\end{equation*}
In the present nonreduced setup the higher rank theta function
$\vartheta(t)$ \eqref{vartheta} can be written in terms
of Jacobi's theta function,
\begin{equation}\label{thetaspecial}
\vartheta(t)=\prod_{i=1}^n\theta\bigl(-q^{\frac{1}{2}}t_i;q\bigr).
\end{equation}
This allows us to rewrite the theta function
factor of $\widehat{c}_\eta(\gamma)$ 
(see Corollary \ref{monicc}) as
\[
\frac{\vartheta((w_0\eta)^{-1}(k_{2a_0}k_{2\theta}^{-1})^{\rho_s^\vee}\gamma)}
{\vartheta((k_{2a_0}k_{2\theta}^{-1})^{\rho_s^\vee}\gamma)}=
\prod_{i=1}^n\frac{\theta(q\eta_i\gamma_i/\widetilde{d})}
{\theta(q\gamma_i/\widetilde{d})}.
\]
The normalized $c$-function $\widehat{c}_\eta(\gamma)$ thus becomes
\begin{equation}\label{cnr}
\widehat{c}_\eta(\gamma)=\prod_{i=1}^n\frac{\bigl(\widetilde{a}\gamma_i^{-1},
\widetilde{b}\gamma_i^{-1},\widetilde{c}\gamma_i^{-1},
\widetilde{d}\gamma_i^{-1}/\eta_i, q\eta_i\gamma_i/\widetilde{d};q\bigr)_{\infty}}
{\bigl(\gamma_i^{-2},q\gamma_i/\widetilde{d};q\bigr)_{\infty}}
\prod_{1\leq i<j\leq n}
\frac{\bigl(k_\vartheta^2\gamma_i^{-1}\gamma_j,
k_\vartheta^2\gamma_i^{-1}\gamma_j^{-1};q\bigr)_{\infty}}
{\bigl(\gamma_i^{-1}\gamma_j,\gamma_i^{-1}\gamma_j^{-1};q\bigr)_{\infty}},
\end{equation}
where we use the shorthand notation
\[
\bigl(\alpha_1,\ldots,\alpha_j;q\bigr)_r:=
\prod_{i=1}^j\bigl(\alpha_i;q\bigr)_r.
\]

In the remainder of this subsection we set $n=1$. Then
$\mathcal{D}$ is the Askey-Wilson \cite{AW} second-order $q$-difference 
operator. In this case the $c$-function expansion
(Corollary \ref{monicc}) 
was proved in \cite{KS,Stexp} using the theory of 
one-variable basic hypergeometric
series. Important ingredients are the 
explicit basic hypergeometric 
series expressions for $\mathcal{E}_+$ and $\widehat{\Phi}_\eta$,
which we now recall. 

The
${}_{r+1}\phi_r$ basic hypergeometric series \cite{GR} is the
convergent series
\begin{equation}\label{rphi}
{}_{r+1}\phi_r\left(\begin{matrix}a_1,a_2,\ldots,a_{r+1}\\
b_1,b_2,\ldots,b_r\end{matrix};q,z\right):=
\sum_{j=0}^{\infty}\frac{\bigl(a_1,a_2,\ldots,a_{r+1};q\bigr)_j}
{\bigl(q,b_1,\ldots,b_r;q\bigr)_{j}}z^j,\qquad |z|<1.
\end{equation}
The very-well-poised ${}_8\phi_7$ basic
hypergeometric series is given by
\[
{}_8W_7(\alpha_0;\alpha_1,\alpha_2,\alpha_3,\alpha_4,\alpha_5;q,z):=
\sum_{r=0}^{\infty}
\frac{1-\alpha_0q^{2r}}{1-\alpha_0}z^r\prod_{j=0}^5\frac{\bigl(\alpha_j;q\bigr)_r}
{\bigl(q\alpha_0/\alpha_j;q\bigr)_r},
\qquad |z|<1.
\] 
Very-well-poised ${}_8\phi_7$ basic 
hypergeometric series solutions of
the eigenvalue equation
\begin{equation}\label{eigvAW}
\mathcal{D}f=(\widetilde{a}(\gamma+\gamma^{-1})-\widetilde{a}^2-1)f
\end{equation}
were obtained in \cite{IR}.
On the other hand we already observed that
$E_+(\gamma_{-r};\cdot)$ ($r\in\mathbb{Z}_{\geq 0}$)
is the inversion invariant, Laurent polynomial 
solution of \eqref{eigvAW} with spectral point 
$\gamma=\gamma_{-r}$, and that both $\mathcal{E}_+(\cdot,\gamma)$ and 
$\widehat{\Phi}_\eta(\cdot,\gamma)$ satisfy \eqref{eigvAW}.
These solutions are related as follows. 
\begin{prop} For the nonreduced case with $n=1$, we have
\begin{equation}\label{AWpol}
E_+(\gamma_{-r};t)={}_4\phi_3\left(\begin{matrix} q^{-r},q^{r-1}abcd,at,a/t\\
ab,ac,ad\end{matrix};q,q\right),\qquad r\in\mathbb{Z}_{\geq 0},
\end{equation}
\begin{equation}\label{AWfun}
\mathcal{E}_+(t,\gamma)=\frac{\Bigl(\frac{qat\gamma}{\widetilde{d}},
\frac{qa\gamma}{\widetilde{d}t},\frac{qa}{d},\frac{q}{ad};q\Bigr)_{\infty}}
{\Bigl(\widetilde{a}\widetilde{b}\widetilde{c}\gamma,
\frac{q\gamma}{\widetilde{d}},\frac{qt}{d},\frac{q}{dt};q\Bigr)_{\infty}}
{}_8W_7\Bigl(\frac{\widetilde{a}\widetilde{b}\widetilde{c}\gamma}{q};
at,\frac{a}{t},\widetilde{a}\gamma,\widetilde{b}\gamma,\widetilde{c}\gamma;
q,\frac{q}{\widetilde{d}\gamma}\Bigr)
\end{equation}
for $|q/\widetilde{d}\gamma|<1$, and
\begin{equation}\label{AWHC}
\widehat{\Phi}_\eta(t,\gamma)=
\frac{\theta\bigl(\frac{ad}{\eta};q\bigr)
\theta\bigl(\frac{q\widetilde{a}t\gamma}{d};q\bigr)}
{\theta\bigl(\frac{\widetilde{d}}{\eta\gamma};q\bigr)
\theta\bigl(\frac{qt}{d};q\bigr)}
\frac{\Bigl(\frac{qa\gamma}{\widetilde{a}t},
\frac{qb\gamma}{\widetilde{a}t},
\frac{qc\gamma}{\widetilde{a}t},
\frac{q\widetilde{a}\gamma}{dt},
\frac{d}{t};q\bigr)_{\infty}}
{\Bigl(\frac{q}{at},
\frac{q}{bt},
\frac{q}{ct},
\frac{q}{dt},
\frac{q^2\gamma^2}{dt};q\Bigr)_{\infty}}
{}_8W_7\Bigl(\frac{q\gamma^2}{dt};
\frac{q\gamma}{\widetilde{a}},
\frac{q\gamma}{\widetilde{d}},
\widetilde{b}\gamma,
\widetilde{c}\gamma,
\frac{q}{dt};
q,\frac{d}{t}\Bigr)
\end{equation}
for $|d/t|<1$.
\end{prop}
\begin{proof}
Formula \eqref{AWpol} follows from 
\cite[Thm. 4.2]{Stexp}, 
\eqref{AWfun} from \cite[Thm. 4.2]{Stexp} and
\eqref{AWHC} from \cite[\S 4]{KS}.
\end{proof}
The proposition shows that $E_+(\gamma_{-r};\cdot)$ 
is the normalized symmetric Askey-Wilson \cite{AW} polynomials 
of degree $r\in\mathbb{Z}_{\geq 0}$, 
and that $\mathcal{E}_+$ coincides up to a
constant multiple with the Askey-Wilson function from \cite{KS}.

In the present nonreduced, rank one setting the $c$-function expansion
\begin{equation}\label{cfunctionexpansionagain}
\mathcal{E}_+(t,\gamma)=\widehat{c}_\eta(\gamma_0)^{-1}
\bigl(\widehat{c}_\eta(\gamma)\widehat{\Phi}_\eta(t,\gamma)+
\widehat{c}_\eta(\gamma^{-1})\widehat{\Phi}_\eta(t,\gamma^{-1})\bigr)
\end{equation}
with
\[
\widehat{c}_\eta(\gamma)=
\frac{\bigl(\widetilde{a}\gamma^{-1},\widetilde{b}\gamma^{-1},
\widetilde{c}\gamma^{-1},\widetilde{d}\gamma^{-1}/\eta,
q\eta\gamma/\widetilde{d};q\bigr)_{\infty}}{\bigl(\gamma^{-2},
q\gamma/\widetilde{d};q\bigr)_{\infty}}
\]
is a special case of Bailey's three term recurrence
relation for very-well-poised ${}_8\phi_7$-series (see \cite[(III.37)]{GR}).
This follows by repeating the proof of \cite[Prop. 1]{KS} (formula 
\eqref{cfunctionexpansionagain}
is more general since it does not involve restriction to a $q$-interval).
See also \cite{StQ} for a detailed discussion.
\begin{rema}
The explicit expressions of 
$\mathcal{E}_+$ and $\widehat{\Phi}_\eta$ as meromorphic functions
on $\mathbb{C}^*\times\mathbb{C}^*$ can be obtained from the
above explicit expressions by writing the ${}_8W_7$ series as 
sum of two balanced 
${}_4\phi_3$ series using Bailey's formula \cite[(III.36)]{GR},
see for instance formula \cite[(3.3)]{KS} for $\mathcal{E}_+$
(the basic ${}_{r+1}\phi_r$ series \eqref{rphi} 
is called balanced if $z=q$ and $qa_1a_2\cdots a_{r+1}=
b_1b_2\cdots b_r$).
\end{rema}

\subsection{The $\textup{GL}_m$ case}\label{GLsection}
In this subsection we 
use the notations from Example \ref{exampleLambda}{\bf (ii)}.
We identify $T\simeq\bigl(\mathbb{C}^*\bigr)^{m}$, taking $t_i:=t^{\epsilon_i}$
as the coordinates ($1\leq i\leq m$). Note that the multiplicity function
$k$ is constant (its constant value will also be denoted by $k$).
The $q$-difference operators
$D_{e_r}$ associated to the elementary symmetric functions
\[
e_r(t)=\sum_{\stackrel{I\subseteq\{1,\ldots,m\}}{\#I=r}}\prod_{j\in I}t_j,
\qquad 1\leq r\leq m
\] 
are Ruijsenaars' \cite{R} quantum Hamiltonians of 
the relativistic quantum trigonometric Calogero-Moser-Sutherland model,
\begin{equation}
D_{e_r}=\sum_{\stackrel{I\subseteq\{1,\ldots,m\}}{\#I=r}}
\left(\prod_{i\in I, j\not\in J}\frac{k^{-1}t_i-kt_j}{t_i-t_j}\right)
\tau\bigl(\sum_{i\in I}\epsilon_i\bigr)_q,\qquad 1\leq r\leq m.
\end{equation}
The monic version of the $c$-function (Corollary \ref{monicc})
becomes
\begin{equation}\label{cetaGL}
\widehat{c}_\eta(\gamma)=\frac{\vartheta((w_0\eta)^{-1}\gamma)}
{\vartheta(\gamma)}
\prod_{1\leq i<j\leq m}\frac{\bigl(k^2\gamma_j/\gamma_i;q\bigr)_{\infty}}
{\bigl(\gamma_j/\gamma_i;q\bigr)_{\infty}}.
\end{equation}
Also in the present $\textup{GL}_m$ case, the higher rank theta functions 
appearing in \eqref{cetaGL}
can be expressed as product of Jacobi
theta functions by \eqref{thetaspecial}.

Our results for $\textup{GL}_2$
can be matched with the extensive literature
on Heine's $q$-analog of the hypergeometric differential equation
(see, e.g., \cite[Chpt. 1]{GR},
\cite[Chpt. 3, \S 1.7]{Kln} and \cite[\S 6.3]{M}). It leads to 
explicit expressions of $E_+(\gamma_\lambda;\cdot)$,
$\Phi$ and $\widehat{\Phi}_\eta$
in terms of Heine's $q$-analog
of the hypergeometric function.
For completeness we detail this link here.

Heine's basic hypergeometric $q$-difference equation is
\begin{equation}\label{qd}
\begin{split}
z(c-abqz)(\partial_q^2u)(z)+\left(\frac{1-c}{1-q}+
\frac{(1-a)(1-b)-(1-abq)}{1-q}z\right)(\partial_qu)&(z)\\
-\frac{(1-a)(1-b)}{(1-q)^2}&u(z)=0,
\end{split}
\end{equation}
with 
\[(\partial_qu)(z):=\frac{u(z)-u(qz)}{(1-q)z}
\] 
the $q$-derivative.
Note that \eqref{qd} formally reduces to 
the hypergeometric differential equation
\[
z(1-z)u^{\prime\prime}(z)+(c-(a+b+1)z)u^\prime(z)-abu(z)=0
\]
by replacing in \eqref{qd} the parameters
$a,b,c$ by $q^a,q^b,q^c$ and taking the limit
$q\rightarrow 1$. A distinguished solution of \eqref{qd} is Heine's
basic hypergeometric function
\begin{equation}\label{Heine}
\phi_H(z):={}_2\phi_1\left(\begin{matrix} a,b\\ c\end{matrix};q,z\right)
\end{equation}
for $|z|<1$.

Note that for $\textup{GL}_2$,
\[\widehat{\mathcal{W}}(t,\gamma)=
\frac{\theta\bigl(-q^{\frac{1}{2}}kt_1/\gamma_2;q\bigr)
\theta\bigl(-q^{\frac{1}{2}}t_2/k\gamma_1;q\bigr)}
{\theta\bigl(-q^{\frac{1}{2}}t_1;q\bigr)\theta\bigl(-q^{\frac{1}{2}}t_2;q\bigr)}.
\]
\begin{lem}\label{tran}
Fix $\gamma\in T\simeq\bigl(\mathbb{C}^*\bigr)^2$.
If $u\in\mathcal{M}(\mathbb{C}^*)$
satisfies Heine's basic hypergeometric $q$-difference
equation \eqref{qd} with the parameters $a,b,c$ given by
\begin{equation}\label{pppp}
a=k^2,\quad b=k^2\gamma_1/\gamma_2,\quad c=q\gamma_1/\gamma_2,
\end{equation}
then the meromorphic function $f(t):=
\widehat{\mathcal{W}}_\eta(t,\gamma)u(qt_2/k^2t_1)$
satisfies
\begin{equation}\label{spgl2}
D_{e_r}f=e_r(\gamma^{-1})f,\qquad r=1,2,
\end{equation}
where the 
$D_{e_r}$ are the $\textup{GL}_2$ Macdonald-Ruijsenaars $q$-difference
operators. Conversely, if $f$ is a meromorphic solution of \eqref{spgl2}
of the form $f(t_1,t_2)=
\widehat{\mathcal{W}}_\eta(t,\gamma)u(qt_2/k^2t_1)$ for some 
$u\in\mathcal{M}(\mathbb{C}^*)$, then $u$ satisfies \eqref{qd} with parameters
$a,b,c$ given by \eqref{pppp}.
\end{lem}
\begin{proof}
Direct computation.
\end{proof}
\begin{cor}[$\textup{GL}_2$ case]
{\bf (i)} The normalized symmetric Macdonald polynomial
$E_+(\gamma_\lambda;t)$ ($\lambda=\lambda_1\epsilon_1+\lambda_2\epsilon_2$
with $\lambda_i\in\mathbb{Z}$ and 
$\lambda_1\leq\lambda_2$) is given by
\begin{equation}\label{GLpol}
E_+(\gamma_\lambda;t)=
k^{\lambda_1-\lambda_2}\frac{\bigl(q^{1+\lambda_1-\lambda_2}/k^2;
q\bigr)_{\lambda_2-\lambda_1}}{\bigl(q^{1+\lambda_1-\lambda_2}/k^4;
q\bigr)_{\lambda_2-\lambda_1}}
t_1^{\lambda_2}t_2^{\lambda_1}
{}_2\phi_1\left(\begin{matrix} k^2,q^{\lambda_1-\lambda_2}\\
q^{1+\lambda_1-\lambda_2}/k^2\end{matrix};q,\frac{qt_2}{k^2t_1}\right).
\end{equation}
{\bf (ii)} The monic basic Harish-Chandra series is explicitly given by
\begin{equation}\label{GLHC}
\widehat{\Phi}_\eta(t,\gamma)=\widehat{\mathcal{W}}_\eta(t,\gamma)
{}_2\phi_1\left(\begin{matrix} k^2,k^2\gamma_1/\gamma_2\\
q\gamma_1/\gamma_2\end{matrix};q,\frac{qt_2}{k^2t_1}\right)
\end{equation}
for $|qt_2/k^2t_1|<1$.
\end{cor}
\begin{proof}
{\bf (i)} 
Consider the solution
\[
\phi_H(z)={}_2\phi_1\left(\begin{matrix} k^2, q^{\lambda_1-\lambda_2}\\
q^{1+\lambda_1-\lambda_2}/k^2\end{matrix};q,z\right)
\] 
of the basic hypergeometric $q$-difference equation \eqref{qd} 
with parameters $a,b,c$ given by \eqref{pppp} 
and with $\gamma=\gamma_\lambda=(q^{\lambda_1}k^{-1},q^{\lambda_2}k)$.
It is a polynomial in $z$ of 
degree $\lambda_2-\lambda_1$ (it is essentially the 
continuous $q$-ultraspherical polynomial). In addition,
$\widehat{\mathcal{W}}(t,\gamma_\lambda)=
q^{-\frac{|\lambda|^2}{2}}t_1^{\lambda_2}t_2^{\lambda_1}$.
By the previous lemma we conclude that 
$P_\lambda^+(t)=t_1^{\lambda_2}t_2^{\lambda_1}\phi_H(qt_2/k^2t_1)$,
see also
\cite[\S 6.3]{M} and \cite[Chpt. 2]{C}. The normalization factor turning
$P_\lambda^+(t)$ into the normalized symmetric Macdonald polynomial
$E_+(\gamma_\lambda;t)$ can for instance be computed using
the $q$-Vandermonde formula \cite[(II.6)]{GR}.\\
{\bf (ii)} 
$\widehat{\mathcal{W}}_\eta(t,\gamma)\phi_H(qt_2/k^2t_1)$
with the parameters $a,b,c$ in $\phi_H$ given by \eqref{pppp}
and $\widehat{\Phi}_\eta(t,\gamma)$ both satisfy \eqref{spgl2} and have 
the same asymptotic expansion for small $|t_2/t_1|$. This forces them
to be equal (cf., e.g., \cite[Chpt. 3, \S 1.7]{Kln} and \cite[Thm. 2.3]{CO}).
\end{proof}

The monic $\textup{GL}_2$ $c$-function expansion 
\begin{equation}\label{monicgl2}
\mathcal{E}_+(t,\gamma)=\widehat{c}_\eta(\gamma_0)^{-1}
\bigl(\widehat{c}_\eta(\gamma)\widehat{\Phi}_\eta(t,\gamma)+
\widehat{c}_\eta(w_0\gamma)\widehat{\Phi}_\eta(t,w_0\gamma)\bigr),
\end{equation}
where $w_0\gamma=(\gamma_2,\gamma_1)$ and
\[
\widehat{c}_\eta(\gamma)=\frac{\theta(-q^{\frac{1}{2}}\gamma_1/\eta_2;q)
\theta(-q^{\frac{1}{2}}\gamma_2/\eta_1;q)}
{\theta(-q^{\frac{1}{2}}\gamma_1;q)\theta(-q^{\frac{1}{2}}\gamma_2;q)}
\frac{\bigl(k^2\gamma_2/\gamma_1;q\bigr)_{\infty}}
{\bigl(\gamma_2/\gamma_1;q\bigr)_{\infty}},
\]
thus yields an explicit expression of the
$\textup{GL}_2$ basic hypergeometric function $\mathcal{E}_+(t,\gamma)$
as sum of two ${}_2\phi_1$ basic hypergeometric series.

Another solution of Heine's basic hypergeometric $q$-difference equation
\eqref{qd} is
\begin{equation*}
\begin{split}
v(x):=&\frac{\theta(ax;q)}{\theta(x;q)}{}_2\phi_1\left(
\begin{matrix} a,qa/c\\ qa/b\end{matrix}; 
q,\frac{qc}{abx}\right)\\
=&\frac{\theta(ax;q)\bigl(a,q^2/bx;q\bigr)_{\infty}}
{\theta(x;q)\bigl(qa/b,qc/abx;q\bigr)_{\infty}}
{}_2\phi_1\left(\begin{matrix}
q/b, qc/abx\\
q^2/bx\end{matrix};q,a\right),
\end{split}
\end{equation*}
see, e.g., \cite[Chpt. 3, \S 1.7]{Kln}
(the second formula follows from Heine's transformation formula 
\cite[(III.1)]{GR}). 
By Lemma \ref{tran} it yields
yet another solution 
\begin{equation*}
f(t)=\widehat{\mathcal{W}}_\eta(t,\gamma)
\frac{\theta(qt_2/t_1;q)}{\theta(qt_2/k^2t_1;q)}
\frac{\bigl(qt_1\gamma_2/t_2\gamma_1;q\bigr)_{\infty}}
{\bigl(qt_1/k^2t_2;q\bigr)_{\infty}}
{}_2\phi_1\left(
\begin{matrix} q\gamma_2/k^2\gamma_1, qt_1/k^2t_2\\
qt_1\gamma_2/t_2\gamma_1\end{matrix}; q,k^2\right)
\end{equation*}
of the system \eqref{spgl2} of $\textup{GL}_2$
Macdonald $q$-difference equations. The various
${}_2\phi_1$ basic hypergeometric
series solutions of \eqref{spgl2} are related by explicit
connection coefficient formulas, see, e.g.,
\cite{StQ}.

We finish this subsection by relating the monic $\textup{GL}_2$ basic
Harish-Chandra series $\widehat{\Phi}_\eta$ \eqref{GLHC}
to the monic nonreduced rank one basic Harish-Chandra series \eqref{AWHC}, 
which we will denote here by $\widehat{\Phi}_\xi^{nr}$. Recall that
in the nonreduced rank one setting
the associated multiplicity function is determined by
the four values $k^{nr}=(k^{nr}_\theta,k^{nr}_{2\theta},
k^{nr}_0,k^{nr}_{2a_0})$. We write $\theta(x_1,\ldots,x_r;q)=
\prod_{j=1}^r\theta(x_j;q)$ for products of Jacobi theta functions.
\begin{prop}\label{propr}
Let $\eta=(\eta_1,\eta_2)\in\bigl(\mathbb{C}^*\bigr)^2$ and 
$\xi\in\mathbb{C}^*$. Let $k^{nr}$ be the multiplicity function
$k^{nr}=(k,k,0,0)$ with $0<k<1$. For $\gamma=(\gamma_1,
\gamma_2)\in\bigl(\mathbb{C}^*\bigr)^2$ set $\gamma^{\pm 2}:=
(\gamma_1^{\pm 2},\gamma_2^{\pm 2})$.
Then 
\begin{equation}\label{relPhi}
\widehat{\Phi}_\eta(t,\gamma^2;k,q)=
C_{\eta,\xi}(t,\gamma)
\widehat{\Phi}_\xi^{nr}\Bigl(\frac{t_1}{t_2},\frac{\gamma_1}{\gamma_2};
k^{nr},q\Bigr)
\end{equation}
with $\widehat{\Phi}_\eta$ the
$\textup{GL}_2$ monic basic Harish-Chandra series \eqref{GLHC} and 
\[
C_{\eta,\xi}(t,\gamma)=
\frac{\theta(-q^{\frac{1}{2}}\eta_1/k,-q^{\frac{1}{2}}\eta_2k,
-q^{\frac{1}{2}}k\gamma_2/\xi\gamma_1,-q^{\frac{1}{2}}kt_1/\gamma_2^2,
-q^{\frac{1}{2}}t_2/k\gamma_1^2,-q^{\frac{1}{2}}t_1/t_2;q)}
{\theta(-q^{\frac{1}{2}}\eta_1/\gamma_2^2,
-q^{\frac{1}{2}}\eta_2/\gamma_1^2,-q^{\frac{1}{2}}k^2/\xi,
-q^{\frac{1}{2}}t_1,-q^{\frac{1}{2}}t_2,
-q^{\frac{1}{2}}kt_1\gamma_1/t_2\gamma_2;q)}.
\]
\end{prop}
\begin{proof}
If the meromorphic function $f(x,z)$ in $(x,z)\in
\mathbb{C}^*\times\mathbb{C}^*$ 
satisfies the Askey-Wilson second-order $q$-difference equation
\[
\mathcal{D}f(\cdot,z)=(\widetilde{a}(z+z^{-1})-\widetilde{a}^2-1)f(\cdot,z)
\]
with respect to the multiplicity function $k^{nr}=(k,k,0,0)$ 
then 
\[
g(t):=C_{\eta,\xi}(t,\gamma)f\Bigl(\frac{t_1}{t_2},
\frac{\gamma_1}{\gamma_2}\Bigr)
\]
satisfies the 
$\textup{GL}_2$ Macdonald-Ruijsenaars $q$-difference equations
\begin{equation}\label{eq}
\bigl(D_{e_r}g\bigr)(t)=e_r(\gamma^{-2})g(t),\qquad r=1,2,
\end{equation}
cf. \cite{StQ}. 
Here we use that the prefactor $C_{\eta,\xi}(t,\gamma)$ satisfies
\[
C_{\eta,\xi}(\tau(-\epsilon_r)_qt,\gamma)=
\gamma_1^{-1}\gamma_2^{-1}C_{\eta,\xi}(t,\gamma)
\]
for $r=1,2$. Hence both sides of \eqref{relPhi} satisfy \eqref{eq}.
In addition, both sides of \eqref{relPhi} have 
an expansion of the form
\[
\widehat{\mathcal{W}}_\eta(t,\gamma^2)\sum_{r=0}^{\infty}\Xi_r\Bigl(\frac{t_2}
{t_1}\Bigr)^r,
\qquad \Xi_0=1
\]
for $|t_2/t_1|$ sufficiently small.
This forces the identity \eqref{relPhi}, cf. the proof of
\eqref{GLHC}.
\end{proof}
A similar statement 
is not true if the role of the basic Harish-Chandra series in Proposition
\ref{propr} is replaced by 
the associated basic hypergeometric functions. 
This follows from a comparison of the associated $c$-function expansions.

\begin{rema}
By \eqref{GLHC} and \eqref{AWHC}, formula \eqref{relPhi} is 
an identity expressing a very-well-poised ${}_8\phi_7$
basic hypergeometric series as a ${}_2\phi_1$ basic hypergeometric series.
After application of the transformation
formula \cite[(III.23)]{GR} to the very-well-poised ${}_8\phi_7$
series, this identity becomes a special case of \cite[(3.4.7)]{GR}.
\end{rema}

\vspace{1cm}
\noindent
{\bf Acknowledgment:} the author was partially supported by 
the Netherlands Organization for Scientific Research (NWO)
via the VIDI-grant ``Symmetry and modularity in
exactly solvable models''.



\end{document}